\documentclass[12pt]{article}
\usepackage{enumerate}
\usepackage{amsmath}
\usepackage{amssymb}
\usepackage{amsthm}
\usepackage{amsfonts}
\usepackage{esint} 
\usepackage{mathrsfs}
\usepackage{mathtools}
\usepackage{tikz}
\usetikzlibrary{shadings,calc}
\usepackage{tikz-cd}
\usepackage{appendix}
\usepackage{relsize}
\usepackage{mathtools}
\usepackage{hyperref}

\usepackage{graphicx}
\numberwithin{equation}{section}

\newcommand{\R}{\mathbb{R}}

\newcommand{\limit}{\rightarrow}

\DeclareMathOperator{\dist}{dist}
\DeclareMathOperator{\capacity}{cap_2}
\DeclareMathOperator{\interior}{int}
\DeclareMathOperator{\osc}{osc}
\DeclareMathOperator{\Lip}{Lip}
\DeclareMathOperator{\supp}{supp}

\newtheorem{theorem}{Theorem}[section]
\newtheorem*{theorem*}{Theorem}
\newtheorem{lemma}[theorem]{Lemma}

\newtheorem{prop}[theorem]{Proposition}

\newtheorem*{Mtheorem}{Main Theorem}

\theoremstyle{definition}
\newtheorem{defn}{Definition}[section]
\newtheorem{obs}{Observation}

\theoremstyle{remark}
\newtheorem*{remark}{Remark}

\newcommand{\pO}{\partial\Omega}
\newcommand{\dt}{\widetilde\Delta}
\newcommand{\Carleson}{\mathscr{C}}
\newcommand{\Carl}[1]{|\nabla{#1}|^2\delta(X)}

\newcommand{\ve}{\varphi_{\epsilon}}
\DeclareMathOperator{\divg}{div}
\DeclareMathOperator{\diam}{diam}
\DeclareMathOperator{\IT}{IT}

\title{BMO Solvability and $A_{\infty}$ Condition of the Elliptic Measures in Uniform Domains}
\author{Zihui Zhao
}
\date{}

\newcommand{\Addresses}{{
  \bigskip
  \footnotesize
  \textsc{DEPARTMENT OF MATHEMATICS, UNIVERSITY OF WASHINGTON, BOX 354350, \newline SEATTLE, WA 98195-435.}\par\nopagebreak
  \textit{E-mail address}: \texttt{zhaozh@uw.edu}
}}

\begin{document}

\makeatletter{\renewcommand*{\@makefnmark}{}
\makeatother}

\maketitle 



\begin{abstract}
	We consider the Dirichlet boundary value problem for divergence form elliptic operators with bounded measurable coefficients. We prove that for uniform domains with Ahlfors regular boundary, the BMO solvability of such problems is equivalent to a quantitative absolute continuity of the elliptic measure with respect to the surface measure, i.e. $\omega_L\in A_{\infty}(\sigma)$. This generalizes a previous result on Lipschitz domains by Dindos, Kenig and Pipher (see \cite{DKP}).
\end{abstract}

\bigskip

\section{Introduction}
Consider a bounded domain $\Omega\subset\mathbb{R}^n$ ($n\geq 3$). Let $L$ be an operator defined as $Lu=-\divg(A\nabla u)$, where $A(X) = \left( a_{ij}(X) \right)_{i,j=1}^n$ is a real $n\times n$ matrix on $\Omega$ that  is bounded measurable and uniformly elliptic: that is, there exists a constant $\lambda>0$ such that
\[ \lambda^{-1} |\xi|^2 \leq \langle A(X)\xi, \xi \rangle \leq \lambda |\xi|^2 \]
for all $\xi\in\mathbb{R}^n\setminus\{0\}$ and $X\in\Omega$. The matrix $A$ is not assumed to be symmetric.

We say $\Omega$ is regular, if for any continuous function $f\in C(\pO)$, the solution $u$ to the Dirichlet problem
\begin{equation}\tag{D}\label{ellp}
	\left\{\begin{array}{ll}
		Lu=0 & \text{on~} \Omega \\
		u=f & \text{on~} \pO
	\end{array}\right.
\end{equation}
is continuous in $\overline\Omega$. In particular $u=f$ on the boundary in the classical sense. It was proved in \cite{LSW} that $\Omega$ is regular for elliptic operator $L$ if and only if it is regular for the Laplacian. Assume $\Omega$ is regular, then by the Riesz representation theorem, for any $X\in\Omega$ there is a probability measure $\omega^X=\omega_L^X$ on $\pO$ such that 
\begin{equation}\label{representation}
	u(X)=\int_{\pO} f(Q) d\omega^X(Q)
\end{equation}
for any boundary value $f\in C(\pO)$ and its corresponding solution $u$. The measure $\omega^X$ is called the elliptic measure (harmonic measure if $L=-\Delta$) of $\Omega$ at $X$. From \eqref{representation}, we see that the information about the boundary behavior of solutions to the Dirichlet problem is encoded in the elliptic measure $\omega^X$. If $\Omega$ is connected, the elliptic measures $\omega^X$ and $\omega^Y$ at different points $X,Y\in\Omega$ are mutually absolute continuous. 
Thus for the problem we are concerned with, we just need to study the elliptic measure at a fixed point $X_0\in\Omega$. Denote $\omega_L=\omega_L^{X_0}$. For Lipschitz domains, by the work of \cite{C}, \cite{HW} and \cite{CFMS}, we know the elliptic measure $\omega_L$ and the solutions $u$ to \eqref{ellp} enjoy numerous properties, most notably \eqref{P1}-\eqref{P6} (see Section \ref{sect:prelim}). These pointwise estimates are important on their own, but they also serve as a toolkit to build global PDE estimates for the solutions to the Dirichlet problem. On the other hand, when the surface measure $\sigma$ is well defined on $\pO$, a natural question is: what is the relationship between the measures $\omega_L$ and $\sigma$? In fact, the answer to this measure-theoretic question is closely related to the boundary behavior of solutions to this PDE problem.

%
%

Let $\Omega$ be a regular domain whose surface measure $\sigma = \mathcal{H}^{n-1} |_{\pO}$ on the boundary is locally finite. Here $\mathcal{H}^{n-1} |_{\pO}$ denotes the $(n-1)$-dimensional Hausdorff measure restricted to $\pO$.
For $1<p<\infty$, we say the problem \eqref{ellp} is solvable in $L^p$ if there exists a universal constant $C$ such that for any continuous boundary value $f$ and its corresponding solution $u$,
\begin{equation}\label{Np}
	\|Nu\|_{L^p(\sigma)} \leq C\|f\|_{L^p(\sigma)}, 
\end{equation}  where $Nu(Q) = \max\{|u(X)|: X\in \Gamma(Q)\} $ is the non-tangential maximal function of $u$ (the definition of $\Gamma(Q)$ is specified in \eqref{def:NTC}). From \eqref{representation} we know $Nu(Q)$ is comparable to the Hardy-Littlewood maximal function $M_{\omega_L}f(Q)$ with respect to $\omega_L$. Provided that $\sigma$ is doubling, the theory of weights tells us
\[ \text{problem} \eqref{ellp} \text{~is~} L^p \text{~solvable, i.e. \eqref{Np} holds} \Longleftrightarrow \omega_L\in B_q(\sigma), \text{~where~} \frac{1}{p}+\frac{1}{q} =1.  \]
(See Section \ref{sect:prelim} for the definition of $B_q$ weights.)
For the Laplacian on Lipschitz domains, Dahlberg \cite{D} proved the harmonic measure $\omega\in B_2(\sigma)$; therefore \eqref{ellp} is solvable in $L^p$ for $2\leq p <\infty$.


When $p=\infty$, \eqref{Np} follows trivially from the maximal principle. However, if $L=-\Delta$ and $\Omega=\mathbb{R}_{+}^{n}$ is the upper half plane, the problem \eqref{ellp} is also solvable in the BMO space, that is, if $f\in BMO(\partial \mathbb{R}_{+}^{n})$, its harmonic extension $u$ has the property that $\mu=x_n|\nabla u|^2 dx$ is a Carleson measure on $\Omega$ (see \cite{FS}, and also Section 4.4 Theorem 3 of \cite{Stein}). In addition, the Carleson measure norm of $\mu$ is equivalent to the BMO norm of $f$. This BMO solvability also holds for Lipschitz domains, if $\mu$ is replaced by $\delta(x)|\nabla u|^2 dx$ and $\delta(x)=\dist(x,\pO)$ (see \cite{FaNe}).

Recall $A_{\infty}(\sigma)$ is a quantitative version of absolute continuity with respect to $\sigma$ (see Definition \ref{def:Ainfty}).  
By the work of Dahlberg, the elliptic measure $\omega_L \in A_{\infty}(\sigma)$ if $L$ is a \textquotedblleft small perturbation\textquotedblright of the Laplacian (see \cite{D2}). Later in \cite{FKP}, the smallness assumption was removed using harmonic analysis methods (see also \cite{F}). More precisely, in \cite{FKP} the authors showed that if $\omega_{L_0}\in A_{\infty}(\sigma)$ and $L_1$ is a perturbation of $L_0$, then $\omega_{L_1}\in A_{\infty}(\sigma)$. Also recall that $A_{\infty}(\sigma)= \cup_{q>1} B_q(\sigma)$, in other words,
\begin{align*}
	\omega_L\in A_{\infty}(\sigma) & \Longleftrightarrow \text{there exists~} q_0>1 \text{~such that~}\omega_L\in B_q(\sigma) \text{~for all~}1<q\leq q_0 \\
	& \Longleftrightarrow \text{problem} \eqref{ellp} \text{~is~}L^p \text{~solvable for all~} p\geq p_0, \text{~where~} \frac{1}{p_0}+\frac{1}{q_0}=1.
\end{align*}
Note that there is some ambiguity with $p_0$: the fact that \eqref{ellp} is not $L^{p_0}$ solvable does not necessarily imply $\omega_L\notin A_{\infty}(\sigma)$.
 A natural question arises: is there a solvability criterion that directly characterizes $\omega_L\in A_{\infty}(\sigma)$? In 2009, Dindos, Kenig and Pipher showed that for Lipschitz domains, the elliptic measure $\omega_L\in A_{\infty}(\sigma) $ if and only if the problem \eqref{ellp} is BMO-solvable, i.e. for any continuous function $f\in C(\pO)$, the Carleson measure of $\delta(X)|\nabla u|^2 dX$ is controlled by the BMO norm of $f$ (see \cite{DKP}).

The Dirichlet problem \eqref{ellp} has been studied on domains less regular than Lipschitz ones, in particular $BMO_1$ domains, non-tangential accessible (NTA) domains and uniform domains. These non-smooth domains arise naturally in free boundary problems and geometric analysis.
Roughly speaking, uniform domains are domains of which each ball centered at the boundary contains an interior ball comparable in size (interior corkscrew condition, see Definition \ref{def:ICC}) and that satisfy a notion of quantified connectivity (Harnack chain condition, see Definition \ref{def:HCC}). NTA domains are uniform domains whose exterior also satisfies the corkscrew condition. It is easy to see
\begin{equation*}
	\text{Lipschitz domains} \subsetneq BMO_1 \text{~domains} \subsetneq \text{NTA domains} \subsetneq \text{uniform domains}.
\end{equation*}

The notion of NTA domain was introduced by Jerison and Kenig in the pioneer work \cite{NTA}, where they proved the set of properties \eqref{P1}-\eqref{P6} hold for the Laplacian on NTA domains. In \cite{HACAD} the authors addressed the $L^p$ solvability of \eqref{ellp} on chord-arc domains, i.e. NTA domains with Ahlfors regular boundary (see Definition \ref{def:ADR}). On the other hand, by the independent work of \cite{A2} and \cite{HM}, certain properties, notably the boundary comparison principle \eqref{P6}, are shown to hold for the Laplacian on uniform domains (the domains studied in the latter paper also require Ahlfors regular boundary). In Section \ref{sect:BR} we show that all the properties \eqref{P1}-\eqref{P6} hold for uniform domains with Ahlfors regular boundary. Moreover, we prove an equivalent characterization of the absolute continuity of elliptic measures by PDE solvability condition:
\begin{Mtheorem}
	For uniform domains with Ahlfors regular boundary, the elliptic measure $\omega_L\in A_{\infty}(\sigma)$ if and only if the Dirichlet problem \eqref{ellp} is BMO solvable.	
\end{Mtheorem}
\noindent By the definition of Lipschitz domains (see Definition 2.2 of \cite{DKP}) and the coarea formula, one can show the boundary of a Lipschitz domain is Ahlfors regular, with constants depending on the (uniform) Lipschitz constant of the domain. Thus our theorem generalizes the result on Lipschitz domains in \cite{DKP}.

\bigskip

In \cite{AHMNT}, the authors characterized the absolute continuity of the harmonic measure from a geometric point of view. 
For a uniform domain with Ahlfors regular boundary, they showed
\begin{equation}
	\omega \in A_{\infty} \Longleftrightarrow \Omega \text{~is an NTA domain} \Longleftrightarrow \pO \text{~is uniform rectifiable}.\label{geochar}
\end{equation} 
We should point out that some of the implications in \eqref{geochar} had been proved earlier in \cite{DJ}, \cite{S} and \cite{HMU}. 

The complement of the four-corner Cantor set inside a (bounded) ball $B_R\setminus\mathcal{C}$ is an example of a uniform domain with Ahlfors regular boundary, whose boundary is  purely unrectifiable. Thus its harmonic measure $\omega\notin A_{\infty}(\sigma)$ by \eqref{geochar}.
%
%
Our study of general elliptic operators raises the following question: If $\Omega=B_R\setminus \mathcal{C}$, is there a uniformly elliptic operator $L=-\divg(A\nabla)$ such that $\omega_L\in A_{\infty}(\sigma)$? More generally, for a uniform domain with purely unrectifiable boundary, is there an operator $L$ such that $\omega_L\in A_{\infty}(\sigma)$? If so, what geometric information does this carry, and can we characterize the corresponding matrix $A$?

In this paper we assume quantified connectivity of the domain, in order to get a quantified characterization of $\omega_L$ with respect to $\sigma$, i.e. $\omega_L\in A_{\infty}(\sigma)$. Recent work by several authors has addressed the question of the relation between $\omega_L$ and $\sigma$ with no connectivity assumption on the domain. We refer interested readers to \cite{HMM}, \cite{AHMMMTV} and \cite{HL}. 
We would also like to bring the reader's attention to the following sufficient condition of $\omega_L\in A_{\infty}(\sigma)$ that is similar to ours. This is an improvement of the main result in \cite{KKiPT} to non-smooth domains:
\begin{theorem*}[\cite{HMT}]
	Let $\Omega \subset \R^n$ be a uniform domain with Ahlfors regular boundary. 
	Then the elliptic measure $\omega_L\in A_{\infty}(\sigma) $, if for any bounded Borel set $E\subset \pO$, the solution $u$ to problem \eqref{ellp} with characteristic boundary data $\chi_E$ satisfies that the Carleson measure of $\delta(X)|\nabla u|^2 dX$ is uniformly bounded, i.e. replace the right hand side of \eqref{MT:Carlest} by a uniform constant.
\end{theorem*}

\bigskip

The plan for this paper is as follows. We state the definitions and main theorem in Section \ref{sect:prelim}. In Section \ref{sect:BR} we establish some preliminary lemmas for the elliptic measure and solutions to problem \eqref{ellp} on uniform domains with Ahlfors regular boundary. In Section \ref{sect:showBMOs} we assume the elliptic measure is an $A_{\infty}$ weight with respect to the surface measure, and show the BMO solvability, i.e. the Carleson measure of $\delta(X)|\nabla u|^2 dX $ is bounded by the BMO norm of the boundary data. Section \ref{sect:showAinfty} is devoted to the proof of $\omega_L\in A_{\infty}(\sigma)$ under the assumption of the BMO solvability for \eqref{ellp}. In Section \ref{sect:converse} we show that if $\omega_L\in A_{\infty}(\sigma)$, the converse to the Carleson measure estimate holds.

\section{Definitions and statement of the main theorem}\label{sect:prelim}
Throughout this paper, we always assume $\Omega\subset\mathbb{R}^n$ is open and bounded, and $n\geq 3$.

\begin{defn}\label{def:ICC}
The domain $\Omega$ is said to satisfy the \textbf{interior corkscrew condition} (resp. exterior corkscrew condition) if  there are $M, R>0$ such that for any $Q\in\pO$, $r\in (0,R)$, there exists a corkscrew point (or non-tangential point) $A=A_r(Q) \in \Omega$ (resp. $A\in \Omega^c$) such that $|A-Q|<r$ and $\delta(A):= \dist(A,\partial\Omega) \geq r/M$.
\end{defn}
Define the non-tangential cone $\Gamma^{\alpha} (Q)$ at $Q\in\pO$ with aperture $\alpha>1$ as follows
\begin{equation}\label{def:NTC} \Gamma^{\alpha}(Q) = \{X\in\Omega: |X-Q|\leq \alpha \delta(X)\}, \end{equation}
and define the truncated cone $\Gamma^{\alpha}_r (Q) = \Gamma^{\alpha}(Q) \cap B(Q,r) $. We will omit the super-index $\alpha$ when there is no confusion.
The interior corkscrew condition in particular implies $\Gamma_r(Q)$ is nonempty as long as the aperture $\alpha \geq M$. The non-tangential maximal function is defined as $Nu(Q) = \sup\{ |u(X)|: X\in \Gamma(Q)\}$, 
and the square function is defined as $Su(Q) = \left(\iint_{\Gamma(Q)} |\nabla u(X)|^2 \delta(X)^{2-n} dX\right)^{1/2}$.
We also consider truncated square function $S_h u(Q)$, where the non-tangential cone $\Gamma(Q)$ is replaced by the truncated cone $\Gamma_h(Q)$.

\begin{defn}\label{def:HCC}
The domain $\Omega$ is said to satisfy the \textbf{Harnack chain condition} if there are universal constants $C>1$ and $C'>0$, such that for every pair of points $A$ and $A'$ in $\Omega$ satisfying
\begin{equation}\label{cond:Lambda}
	\Lambda:=\frac{|A-A'|}{\min\{\delta(A), \delta(A')\}} > 1, 
\end{equation} 
 there is a chain of open Harnack balls $B_1, B_2, \cdots, B_M$ in $\Omega$ that connects $A$ to $A'$. Namely, $A\in B_1$, $A'\in B_M$, $B_j\cap B_{j+1}\neq\emptyset$ and 
\begin{equation}\label{preHarnackball}
 	C^{-1}\diam(B_j) \leq \delta(B_j) \leq C\diam(B_j)
\end{equation}
     for all $j$.
    Moreover, the number of balls $ M\leq  C'\log_2 \Lambda$. 
\end{defn}

\begin{remark}
\begin{enumerate}
	\item If two points $A, A'\in\Omega$ do not satisfy \eqref{cond:Lambda}, we can simply take the balls $B(A, \delta(A)/2)$ and $B(A',\delta(A')/2)$ to connect them.
	\item We often want the Harnack balls to satisfy $\delta(B_j)> \diam(B_j)$ to be able to enlarge them. This is possible. In fact, if $\Omega$ satisfy the Harnack chain condition, then for any $C_1>1$, there is a constant $C_2$ such that $\Omega$ satisfies the above Harnack chain condition with the comparable size condition \eqref{preHarnackball} replaced by
	\begin{equation}
		C_1\diam(B_j) \leq \delta(B_j)\leq C_2\diam(B_j) \tag{HB} \label{Harnackball}.
	\end{equation} 
	The number of balls $M$ may increase, but is still of the order $\log_2 \Lambda$.
	Moreover, the ratio between $C_2$ and $C_1$ is fixed: $C_2/C_1 \approx C^2(C+1)^2$
	Balls satisfying the condition \eqref{Harnackball} are called \textbf{Harnack balls} with constants $(C_1,C_2)$.
\end{enumerate}	
\end{remark}

\begin{defn}
If $\Omega$ satisfies (1) the interior corkscrew condition and (2) the Harnack chain condition, then we say $\Omega$ is a \textbf{uniform domain}. If in addition, $\Omega$ satisfies the exterior corkscrew condition, we say it is an \textbf{NTA (non-tangential accessible) domain}.
\end{defn}

For any $Q\in\pO$ and $r>0$, let $\Delta=\Delta(Q,r)$ denote the surface ball $B_r(Q) \cap \pO$, and $T(\Delta)=B_r(Q)\cap\Omega$ denote the Carleson region above $\Delta$. We always assume $r< \diam \Omega$.

\begin{defn}\label{def:ADR}
We say that the boundary of $\Omega$ is \textbf{Ahlfors regular} if there exist constants $C_2>C_1>0$, such that for any $Q\in\pO$ and any radius $r>0$,
\[ C_1 r^{n-1} \leq \sigma(\Delta(Q, r)) \leq C_2 r^{n-1}, \]
where $\sigma = \mathcal{H}^{n-1}|_{\pO}$ is the surface measure.
\end{defn}
\begin{defn}\label{def:Ainfty}
Let $\mu$ and $\nu$ be two finite Borel measures on $\pO$. We say $\mu\in A_{\infty}(\nu)$ if for any $\epsilon>0$, there exists $\delta=\delta(\epsilon)>0$ such that for any surface ball $\Delta$ and any $E\subset \Delta$: $\mu(E)/\mu(\Delta) <\epsilon$ whenever $\nu(E)/\nu(\Delta) < \delta$
\end{defn}

The $A_{\infty}$ property is symmetric: $\mu\in A_{\infty}(\nu)$ if and only if $\nu\in A_{\infty}(\mu)$. 
If this holds, there is some $q>1$ such that the kernel function $k=d\mu/d\nu$ satisfies a reverse H\"older inequality: $\left(\fint_{\Delta} k^q d\nu\right)^{1/q} \leq C \left( \fint_{\Delta} k d\nu\right)$ for all surface balls $\Delta$. 
In other words, $A_{\infty}(d\nu) = \cup_{q>1} B_q(d\nu)$.

\begin{Mtheorem}\label{Thm:main}
	Let $\Omega\subset\mathbb{R}^n$ be a uniform domain with Ahlfors regular boundary and the operator $L=-\divg(A\nabla)$, where $A(X)$ is a real $n\times n$ matrix that is bounded and uniformly elliptic on $\Omega$. Then the elliptic measure $\omega_L\in A_{\infty}(\sigma)$ if and only if the problem \eqref{ellp} is BMO solvable, namely the following statement holds: there exists a constant $C>0$ such that for any continuous function $f\in C(\pO)$, if $Lu=0$ in $\Omega$ with $u=f$ on $\pO$, then $\delta(X)|\nabla u|^2 dX$ is a Carleson measure, and 
	\begin{equation}\label{MT:Carlest}
		\sup_{\Delta\subset\pO} \frac{1}{\sigma(\Delta)} \iint_{T(\Delta)} \Carl{u} dX \leq C \|f\|_{BMO(\sigma)}^2.
	\end{equation}
	Moreover, the reverse inequality to \eqref{MT:Carlest} also holds, that is, there exists a constant $C'>0$ such that
	\[ \|f\|_{BMO(\sigma)}^2 \leq C' \sup_{\Delta\subset\pO} \frac{1}{\sigma(\Delta)} \iint_{T(\Delta)} \Carl{u} dX . \]
\end{Mtheorem}
\noindent From here onwards, we write $\omega = \omega_L$ if there is no confusion as to which elliptic operator $L$ we are talking about.

Before we start the proof, we make the following observation: the Carleson measure norm of $\delta(X)|\nabla u|^2 dX$ is in some sense equivalent to the integral of the truncated square function.
Suppose $\Delta=\Delta(Q_0,r)$ is an arbitrary surface ball. For any $X\in T(\Delta)$, we define $\Delta^X=\{Q\in\pO: X\in \Gamma(Q)\}$. Let $Q_X\in\pO$ be a point such that $|X-Q_X|=\delta(X)$. Then 
\begin{equation}
	\Delta(Q_X, (\alpha-1)\delta(X))\subset \Delta^X\subset \Delta(Q_X, (\alpha+1)\delta(X)). \label{DeltaX}
\end{equation} 
Since $\pO$ is Ahlfors regular, \eqref{DeltaX} implies $\sigma(\Delta^X) \approx \delta(X)^{n-1}$. Thus
\begin{align}
 	\iint_{T(\Delta)} \Carl{u} dX & \approx \iint_{T(\Delta)} |\nabla u|^2 \delta(X)^{2-n} \sigma(\Delta^X ) dX \nonumber \\
 	& =\iint_{T(\Delta)} |\nabla u|^2 \delta(X)^{2-n} \int_{ \Delta^X} d\sigma(Q) dX. \label{eq:CarlesonSquarepre}
\end{align} 
Changing the order of integration, on one hand,
\begin{align}
  \iint_{T(\Delta)} |\nabla u|^2 \delta(X)^{2-n} \int_{\Delta^X} d\sigma(Q) dX & \leq \int_{|Q-Q_0|<(\alpha+1)r}\iint_{\Gamma_{\alpha r}(Q)} |\nabla u|^2 \delta(X)^{2-n} dXd\sigma \nonumber \\
  & \leq \int_{(\alpha+1)\Delta}S_{\alpha r}^2(u)d\sigma. \label{eq:CarlesonSquareU}
\end{align}
On the other hand,
\begin{align}
  	\iint_{T(\Delta)} |\nabla u|^2 \delta(X)^{2-n} \int_{ \Delta^X} d\sigma(Q) dX & \geq \int_{|Q-Q_0|<r/2}\iint_{\Gamma_{r/2}(Q)} |\nabla u|^2 \delta(X)^{2-n} dXd\sigma \nonumber \\
  & \geq \int_{\Delta/2}S_{r/2}^2(u)d\sigma,  \label{eq:CarlesonSquareL}
\end{align}
where $\Delta/2 = \Delta(Q_0,r/2)$.
Therefore for any $Q_0\in\pO$,
\begin{equation}
	\sup_{\substack{\Delta=\Delta(Q_0,s) \\ s>0}} \frac{1}{\sigma(\Delta)} \iint_{T(\Delta)} \Carl{u} dX \approx \sup_{\substack{\Delta=\Delta(Q_0,r) \\ r>0}} \frac{1}{\sigma(\Delta)} \int_{\Delta} S_{r}^2(u) d\sigma
\end{equation} 

\section{Preliminaries: boundary behavior of non-negative solutions and the elliptic measure}\label{sect:BR}
The boundary H\"older regularity of the solution to Dirichlet problem has been proved for NTA domains in \cite{NTA}. By tools from potential theory, we show the boundary H\"older regularity holds for uniform domains with Ahlfors regular boundary. The author is aware of Theorem \ref{ARCDC} and \ref{thm:osc} through discussions with Prof. Tatiana Toro, for which the author is grateful.
\begin{defn}
	A domain $\Omega$ is \textbf{Wiener regular}, if for any Lipschitz function $f\in\Lip(\pO)$, there exists a solution $u\in W^{1,2}(\Omega)\cap C(\overline\Omega)$ to $Lu=0$ with Dirichlet boundary data $f$.
\end{defn}

\begin{theorem}
	$\Omega$ is Wiener regular if and only if for any $Q\in\pO$,
	\begin{equation}\label{capacityinf}
		\int_0^{*} \dfrac{\capacity(B_r(Q) \cap \Omega^c)}{r^{n-2}} \frac{dr}{r} = +\infty.
	\end{equation}
	For any set $K$, the capacity is defined as follows:
	\begin{equation}
		\capacity(K) = \inf \bigg\{\int |\nabla \varphi|^2 dx: \varphi \in C_c^{\infty}(\mathbb{R}^n), K\subset \interior\{\varphi\geq 1\}\bigg\}.
	\end{equation}
\end{theorem}
The following condition has been explored extensively by Aikawa (the condition has been mentioned without name in the work of \cite{Ancona}). See \cite{A1} \cite{A2} \cite{A3} for example. 
\begin{defn}
A domain $\Omega$ is said to satisfy the \textbf{capacity density condition (CDC)} if there exist constants $C_0, R>0$ such that 
\begin{equation}
	\capacity(B_r(Q) \cap \Omega^c) \geq C_0 r^{n-2},\quad \text{for any } Q\in\pO \text{ and any } r\in (0,R).
\end{equation}	
\end{defn}
Clearly if the domain $\Omega$ satisfies the CDC, it satisfies \eqref{capacityinf}, thus $\Omega$ is Wiener regular. What is relevant in our case is the following:
\begin{theorem}\label{ARCDC}
	If the domain $\Omega$ has Ahlfors regular boundary, it satisfies the CDC. In particular, $\Omega$ is Wiener regular.
\end{theorem}
\begin{proof}
	For any set $E$ contained in a ball of radius $r$, its Hausdorff content $\mathcal{H}_{\infty}^{n-1}(E)$ satisfies 
	\begin{equation}
		\mathcal{H}_{\infty}^{n-1}(E) \leq Cr^{n/2}\capacity(E)^{1/2},\label{eq:capHcontent}
	\end{equation} 
		where  
	\begin{equation}
		\mathcal{H}_{\infty}^{n-1}(E)= \inf \bigg\{\sum_{i} \left(\frac{\diam B_i}{2}\right)^{n-1}: E\subset \bigcup\limits_i B_i, B_i\text{'s are balls in~}\mathbb{R}^n\bigg\}. \label{def:Hcontent}
	\end{equation} 
	 For the proof see \cite{Evans} page 193-194. 
	 
	 We claim that for an Ahlfors regular boundary $\pO$,
	 \begin{equation}
	 	\mathcal{H}_{\infty}^{n-1}(B_r(Q)\cap\pO)\approx \mathcal{H}^{n-1}(B_r(Q)\cap\pO) \approx r^{n-1}. \label{eq:HHcontent}
	 \end{equation}	
	Let $E=B_r(Q) \cap \pO$. It is clear from the definition \eqref{def:Hcontent} that $\mathcal{H}_{\infty}^{n-1}(E) \leq \mathcal{H}^{n-1}(E)$. Let $\{B_i=B(x_i, r_i)\}$ be an arbitrary covering of $E$. 
	We may assume $x_i\in \pO$; if not, there is some $x'_i\in B_i \cap E$ and we may replace $B_i$ by $B'_i=B(x'_i,3r_i) $. 
	By the countable sub-additivity of $\mathcal{H}^{n-1}$ and Ahlfors regularity of $\pO$, we have
	\[ \mathcal{H}^{n-1}(E) \leq \sum_i \mathcal{H}^{n-1}(B_i\cap\pO) \lesssim \sum_i r_i^{n-1}. \]
	This is true for any covering $E$, hence is true for the infimum. By the definition \eqref{def:Hcontent}
	\[  \mathcal{H}^{n-1}(E) \lesssim \mathcal{H}_{\infty}^{n-1}(E). \] Therefore 
	\[  \mathcal{H}_{\infty}^{n-1}(B_r(Q)\cap\pO) \approx \mathcal{H}^{n-1}(B_r(Q)\cap\pO) \approx r^{n-1}.  \]

	Combining \eqref{eq:capHcontent} and \eqref{eq:HHcontent}, we get 
	\[ \capacity(B_r(Q)\cap\pO) \gtrsim r^{-n} \left(\mathcal{H}_{\infty}^{n-1}(B_r(Q)\cap\pO)\right)^2 \gtrsim r^{n-2}. \]
	Therefore $\capacity(B_r(Q)\cap\Omega^c) \geq \capacity(B_r(Q)\cap\pO) \gtrsim r^{n-2}$, and the domain $\Omega$ satisfies CDC.
\end{proof}

\begin{theorem}[Theorem 6.18 in \cite{Finnish}]\label{thm:osc}
	Consider an open and bounded set $D$. Suppose $u, f\in W^{1,2}(D)\cap C(\overline D)$ are functions satisfying
	\begin{equation*}
		\left\{\begin{array}{ll}
			Lu=0\text{~on~} D &  \\
			u=f \text{~on~}\partial D, & \text{i.e.~} u-f \in W_0^{1,2}(D).
		\end{array}\right.
	\end{equation*}
	Then for any $Q_0\in\partial D$, any $0<\rho< \diam D/2$ and $r\leq \rho$, we have
	\begin{equation}
		\underset{B(Q_0,r)\cap D}{\osc} u  \leq \underset{\overline{B(Q_0,2\rho)}\cap \partial D}{\osc} f + \underset{\partial D}{\osc} f \cdot \exp\left(-C \int_r^{\rho} \frac{\capacity(B_t(Q_0)\cap D^c)}{t^{n-2}} \frac{dt}{t}\right),
	\end{equation}
	where $C=C(n,\lambda)$ is a positive constant, and $\osc_E(g)$ is the oscillation of the function $g$ on a set $E$: $\osc_E(g) = \sup_{x,y\in E} |g(x)-g(y)|$.
\end{theorem}
This theorem, together with the CDC condition, implies the boundary H\"older regularity.
\begin{prop}[Boundary regularity]\label{prop:bHr}
	Let $\Omega$ be a domain satisfying CDC with constants $(C_0,R)$.
	For any $Q_0\in \pO$ and $0<\rho\leq\min\{\diam (\Omega)/4, 3R/4\}$, suppose $u\in W^{1,2}(B_{3\rho}(Q_0)\cap \Omega) \cap C(\overline{B_{3\rho}(Q_0)\cap \Omega})$ satisfies
	\begin{equation*}
		\left\{\begin{array}{ll}
			Lu = 0 & \text{on~} B_{3\rho}(Q_0)\cap \Omega \\
			u=0 & \text{on~} B_{3\rho}(Q_0)\cap \pO.
		\end{array}\right.
	\end{equation*}
	Then
	\begin{equation}\tag{P1}\label{P1}
		|u(X)| \leq 2\left( \frac{|X-Q_0|}{\rho}\right)^{\beta} \sup_{\overline{B(Q_0,3\rho)\cap\Omega}} |u|, \qquad\text{for any } X\in B_{\rho}(Q_0)\cap \Omega.
	\end{equation}
	Here $\beta = \beta(C_0,n, \lambda)$ is a constant in the interval $(0,1]$.
\end{prop}
\begin{proof}
	Let $D= B_{3\rho}(Q_0)\cap\Omega$, we can apply Theorem \ref{thm:osc} to the function $u$ on $D$. For any $X\in B_{\rho}(Q_0)\cap \Omega$, let $r$ be such that $3r/4=|X-Q_0|<\rho$. Then
	\begin{align*}
		|u(X)|  \leq \underset{B(Q_0,r)\cap D}{\osc} u & \leq \underset{\overline{B\left(Q_0,8\rho/3\right)}\cap \partial D}{\osc} u + \underset{\partial D}{\osc} u \cdot \exp\left(-C \int_r^{4\rho/3} \frac{\capacity(B_t(Q_0)\cap \Omega^c)}{t^{n-2}} \frac{dt}{t}\right) \\
		& \leq 0 + \left( 2 \sup_{\partial B(Q_0,3\rho)\cap \Omega} |u| \right) \cdot \exp\left(-C \int_r^{4\rho/3} \frac{\capacity(B_t(Q_0)\cap \Omega^c)}{t^{n-2}} \frac{dt}{t}\right).
	\end{align*}
	Since $\Omega$ satisfies the CDC with constants $(C_0,R)$, and $\rho \leq 3R/4$, we get
	\begin{align*}
		\exp\left(-C \int_r^{4\rho/3} \frac{\capacity(B_t(Q_0)\cap \Omega^c)}{t^{n-2}} \frac{dt}{t}\right) \leq \exp\left(-CC_0 \int_r^{4\rho/3} \frac{dt}{t}\right) \leq \left(\frac{|X-Q_0|}{\rho}\right)^{\beta},
	\end{align*}
	where $\beta= \min\{CC_0, 1\}$. Therefore
	\[ |u(X)| \leq 2\left(\frac{|X-Q_0|}{\rho}\right)^{\beta}\sup_{\overline{B(Q_0,3\rho)\cap \Omega}} |u|. \]
\end{proof}


After we established the boundary regularity of non-negative solutions (Proposition \ref{prop:bHr}), we can prove many other properties of the solutions and the elliptic measure. We quote them as lemmas here. If the matrix $A$ is symmetric, the proof is the same as the case of NTA domains (see \cite{CBMS} or \cite{NTA}); if $A$ is non-symmetric, see \cite{KKPT} Theorem 1.11, (1.12), (1.13) and (1.14). 

\begin{lemma}\label{lm:lbomega}
	Let $\Omega$ be a uniform domain with Ahlfors regular boundary.Then there exists $M_0>0$ such that for any $Q_0\in \pO, s\in (0,R)$, we have 
	\begin{equation}\tag{P2}\label{P2}
		\omega^X(\Delta_s(Q_0)) \geq M_0, \text{~for any~} X\in B\left(A_s(Q_0), \frac{s}{2M}\right).
	\end{equation} 
	Here $M$ and $R$ are the constants mentioned in the definitions of the interior corkscrew condition and CDC, respectively.
\end{lemma}
\begin{lemma}[Boundary Harnack principle]\label{lm:bHp}
	Let $u$ be a non-negative solution in $B_{4s}(Q_0)\cap \Omega$ with vanishing boundary data  on $\Delta_{4s}(Q_0)$. Then
	\begin{equation}\tag{P3}\label{P3}
		u(X) \leq Cu(A_{s}(Q_0)) \quad \text{~for any~}X\in B_s(Q_0)\cap \Omega.
	\end{equation} 
\end{lemma}
\begin{lemma}
	Suppose $X\in \Omega\setminus B_{2s}(Q_0)$, then
	\begin{equation}\tag{P4}\label{P4}
		\dfrac{\omega^X(\Delta_s(Q_0))}{s^{n-1}} \approx \dfrac{G(X,A_s(Q_0))}{s}. 
	\end{equation} 
\end{lemma}

\begin{lemma}[Doubling measure]
There is a constant $C>1$ such that, for any surface ball $\Delta=\Delta(Q_0,r)$, let $2\Delta=\Delta(Q_0,2r)$ be its doubling surface ball, we have
\begin{equation}
	\omega(2\Delta) \leq C\omega(\Delta). \tag{P5}\label{P5}
\end{equation} 
\end{lemma}

In \cite{NTA}, the authors proved the boundary comparison principle \eqref{P6} on NTA domains  using the maximal principle, the doubling of harmonic measure and P. Jones' geometric localization theorem (\cite{Jones}). The geometric localization theorem says the following: if $\Omega$ is an NTA domain, there are constants $C>1$ and $R>0$ such that for any $Q_0\in\pO$ and $r<R$, we can find a domain $\Omega_{Q_0,r}$ such that
\[ B(Q_0,r)\cap\Omega \subset \Omega_{Q_0,r} \subset B(Q_0,Cr)\cap\Omega, \]
and moreover $\Omega_{Q_0,r}$ is an NTA domain, whose NTA constants only depend on the NTA constants of $\Omega$. To prove the boundary comparison principle in our setting, and later to prove Theorem \ref{thm:BMOGreen}, we need a similar geometric localization theorem  for uniform domains with Ahlfors regular boundary. Jones' construction works for uniform domains, but does not have good property on the boundary. To take full advantage of Ahlfors regularity in our setting, we use the Carleson box construction in \cite{HM}.

\begin{theorem}[Geometric localization theorem]\label{thm:GLT}
 Let $\Omega$ be a uniform domain with Ahlfors regular boundary. There are constants $C>1$ and $R>0$ such that for any $Q_0\in \pO$ and $r<R$, there exists a domain $\Omega_{Q_0,r}$ such that
	\[ B(Q_0,r)\cap\Omega \subset \Omega_{Q_0,r} \subset B(Q_0,Cr)\cap\Omega. \]
And moreover, $\Omega_{Q_0,r}$ is a uniform domain with Ahlfors regular boundary, where the constants only depend on the corresponding constants for $\Omega$.
\end{theorem}

\begin{proof}
	Since $\pO\subset\mathbb{R}^n$ is Ahlfors regular, it has a nested grid 
	\[ \mathbb{D}=\{\mathcal{Q}_j^k: k\in\mathbb{Z}, j\in \mathcal{J}_k\} \]
	such that $\pO=\cup_{j\in\mathcal{J}_k}\mathcal{Q}_j^k$ for any $k\in\mathbb{Z}$, and each $\mathcal{Q}_j^k$ has diameter less than $2^{-k}$ and contains a ball of radius comparable to $2^{-k}$. (See Lemma 1.15 in \cite{HM} for details). 	For any surface ball $\Delta=\Delta(Q_0,r)$, there exists $\mathcal{Q}\in \mathbb{D} $ such that $\Delta\subset\mathcal{Q}$ and $l(\mathcal{Q})\approx r$. Let $T_{\mathcal{Q}}$ be the Carleson box defined as (3.52) (see also (3.47), (3,43)) in \cite{HM}). The authors of \cite{HM}  proved in Lemma 3.61 that $T_{\mathcal{Q}}$ is a 1-sided NTA domain (i.e. uniform domain) with Ahlfors regular boundary; and in Lemma 3.55 they proved there is some $c<1$ such that $B(Q_0, cr)\cap\Omega\subset T_{\mathcal{Q}}$.
	From the definition of $T_{\mathcal{Q}}$ and the fact that $l(\mathcal{Q})\approx r$, it is clear that we also have $T_{\mathcal{Q}}\subset B(Q_0, Cr)\cap\Omega$.
	
\end{proof}

\begin{lemma}[Boundary comparison principle]
	Let $u$ and $v$ be non-negative solutions in $B_{4s}(Q_0)\cap\Omega$ with vanishing boundary data on $\Delta_{4s}(Q_0)$. Then
	\begin{equation}\tag{P6}\label{P6}
		\frac{u(X)}{v(X)} \approx \frac{u(A_s(Q_0))}{v(A_s(Q_0))} \quad\text{~for any~} X\in B_s(Q_0)\cap\Omega. 
	\end{equation} 
\end{lemma}
\begin{proof}
	Once we have the geometric localization theorem, the proof proceeds the same as Lemma (4.10) in \cite{NTA}.
\end{proof}

\section{From $\omega_L \in A_{\infty}(\sigma)$ to the Carleson measure estimate}\label{sect:showBMOs}

Assume $\omega \in A_{\infty}(\sigma)$. For any continuous function $f\in C(\pO)$, let $u$ be the solution to the elliptic problem $Lu=0$ with boundary data $f$, we want to show that $|\nabla u|^2 \delta(X)$ is a Carleson measure, and in particular, 
\begin{equation}
	\sup_{\Delta \subset \pO} \frac{1}{\sigma(\Delta )} \iint_{T(\Delta )}|\nabla u|^2 \delta(X) dX \leq C\|f\|_{BMO(\sigma) }^2. \label{eq:CarlesonBMO}
\end{equation} 

Let $\Delta=\Delta(Q_0,r)$ be any surface ball. Denote the constant $c=\max\{\alpha+1, 27\}$ and let $\dt=c\Delta = \Delta(Q_0, cr)$ be its concentric surface ball. 
Let 
\[ f_1= (f-f_{\dt})\chi_{\dt}, \quad f_2 = (f-f_{\dt})\chi_{\pO\setminus\dt}, \quad f_3 = f_{\dt} = \fint_{\dt} f d\sigma ,\]
and let $u_1, u_2, u_3$ be the solutions to $Lu=0$ with boundary data $f_1, f_2, f_3$ respectively. Clearly $u_3$ is a constant, so its Carleson measure is trivial. 

We have shown in Section \ref{sect:prelim} (See \eqref{eq:CarlesonSquarepre} and \eqref{eq:CarlesonSquareU}) that 
\[ \iint_{T(\Delta)} \Carl{u_1} dX \leq C\int_{(\alpha+1)\Delta} S_{\alpha r}^2(u_1) d\sigma.\]
Since $\dt = c\Delta \supset (\alpha+1)\Delta$, it follows that
\begin{equation}
	\iint_{T(\Delta)} \Carl{u_1} dX \leq C \int_{\dt} S_{\alpha r}^2(u_1) d\sigma.  \label{eq:uoneCarlesonSquare}
\end{equation} 
By H\"older inequality, for $p>2$
\begin{equation}
	\int_{\dt}S_{\alpha r}^2(u_1)d\sigma  \leq \sigma(\dt)^{1-\frac 2 p} \left(\int_{\dt}S^p(u_1)d\sigma\right)^{2/p} \leq \sigma(\dt)^{1-\frac 2 p} \|S(u_1)\|_{L^p(\sigma)}^2. \label{eq:SquareHolder}
\end{equation}	

Under the assumption $\omega \in A_{\infty}(\sigma)$, the following theorems are at our disposal:

\begin{theorem}[\cite{CBMS} Theorem 1.4.13(vii) and Lemma 1.4.2]\label{thm:Nu}
Assume $\omega\in B_q(\sigma)$ for some $1<q<\infty$, then the elliptic problem $Lu=0$ is $L^p-$solvable with $1/p+1/q=1$: that is, if $u$ is a solution with boundary value $f\in L^p(\sigma)$, then $\|Nu\|_{L^p(\sigma)}\leq C\|f\|_{L^p(\sigma)}$.	
\end{theorem}

\begin{theorem}[\cite{CBMS} Theorem 1.5.10]\label{thm:NuSu}
Assume $\omega \in A_{\infty} (\sigma)$, then if $Lu = 0$ with boundary value $f$, 
	we have $\|S(u)\|_{L^p(\sigma)}\leq C\|Nu\|_{L^p(\sigma)}$ for any $0<p<\infty$.
\end{theorem}

Apply Theorem \ref{thm:Nu} and Theorem \ref{thm:NuSu} to $u_1$, and get
\begin{equation}
	\|S(u_1)\|_{L^p(\sigma)} \leq C\|f_1\|_{L^p(\sigma)} = C\left(\int_{\dt} |f-f_{\dt}|^p d\sigma\right)^{1/p}. \label{eq:SquareLp}
\end{equation} 
Combining \eqref{eq:uoneCarlesonSquare}, \eqref{eq:SquareHolder} and \eqref{eq:SquareLp}, we get
\begin{align}
	\iint_{T(\Delta)} \Carl{u_1} dX 
	  \leq C\sigma(\Delta) \|f\|_{BMO(\sigma)}^2. \label{uone}
\end{align}

To show similar estimate for $u_2$, let $\{E_m\}$ be a Whitney decomposition of $T(\Delta)$. On each Whitney cube $E_m$, we have the following Cacciopoli type estimate,
\begin{align*}
	 \iint_{E_m} \Carl{u_2} dX & \lesssim \delta(E_m)\iint_{E_m} |\nabla u_2|^2 dX \\
	 & \lesssim \delta(E_m)\cdot \frac{1}{\delta(E_m)^2} \iint_{\frac{3}{2}E_m} |u_2(X)|^2 dX \\
	 & \lesssim \iint_{\frac{3}{2}E_m} \frac{|u_2(X)|^2}{\delta(X)} dX.
\end{align*}
Summing up, we get
\begin{align}\label{eq:utwo}
	\iint_{T(\Delta)} \Carl{u_2} dX 
	& \lesssim \sum_m \iint_{\frac 3 2 E_m} \frac{|u_2(X)|^2}{\delta(X)} dX \nonumber \\
	& \lesssim \iint_{T(\frac 3 2 \Delta)} \frac{|u_2(X)|^2}{\delta(X)} dX.
\end{align}
Recall $3\Delta/2$ denotes $\Delta(Q_0, 3r/2)$, and $T(3\Delta/2)$ denotes $B(Q_0, 3r/2)\cap\Omega$.

Let $u_2^{\pm}$ be the solutions to $Lu=0$ with non-negative boundary data $f_2^{\pm}$, then $u_2 = u_2^+ - u_2^-$ and $|u_2|= u_2^+ + u_2^-$. Let
\begin{equation}
	v(X) = |u_2(X)| = \int_{\pO} \left(f_2^+ + f_2^-\right) d\omega^X = \int_{\pO\setminus\dt} |f-f_{\dt}| d\omega^X. \label{def:v}
\end{equation} 
We have the following lemma:
\begin{lemma}\label{lm:vBMObd}
The function $v$ defined in \eqref{def:v} satisfies
	\begin{itemize}
		\item $v(X) \leq C \|f\|_{BMO(\sigma)} $ for all $X\in T(9\Delta) $. 
		\item $ v(X) \leq C\left(\dfrac{\delta(X)}{r}\right)^\beta \|f\|_{BMO(\sigma)}$ for all $X\in T(3\Delta/2)$. Here $\beta\in (0,1)$ is the degree of boundary H\"older regularity for non-negative solutions, and it only depends on $n$ and the ellipticity of $A$ (see \eqref{P1}).
	\end{itemize}
\end{lemma}

\begin{proof}
	By the definition \eqref{def:v}, the function $v$ vanishes on $\dt$. Note that $\dt\supset 27\Delta$ by the choice of $\dt$, $v$ is a non-negative solution in $T(27\Delta)$ and vanishes on $27\Delta$. Let $A$ be a corkscrew point in $T(9\Delta)$, by Lemma \ref{lm:bHp}
	\[ v(X) \leq C v(A),\quad\text{for all~} X\in T\left( 9\Delta\right). \]
	
	Let $\bar k = d\omega/d\sigma$ be the Radon-Nikodym derivative of $\omega$ with respect to $\sigma$. By the assumption $\omega\in A_{\infty}(\sigma)$, there exists some $q>1$ such that for any surface ball $\Delta'$,
	\begin{equation}\label{eq:Bq}
		 \left( \fint_{\Delta'} \bar{k}^q d\sigma\right)^{1/q} \leq C \fint_{\Delta'} \bar{k} d\sigma ,
	\end{equation}	 
	Let $K(X,\cdot) = d\omega^X/d\omega$ be the Radon-Nikodym derivative of $\omega^X$ with respect to $\omega$ , i.e.
	\begin{equation}\label{def:kernalK}
		K(X,Q) = \lim_{\Delta' \limit Q} \dfrac{\omega^X(\Delta')}{\omega(\Delta')}. 
	\end{equation} 
	Then
	\begin{align} \label{eq:vexp}
		v(A) = \int_{\pO\setminus \dt} |f-f_{\dt}| d\omega^A & = \int_{\pO\setminus \dt} |f-f_{\dt}| K(A,Q)\bar{k}(Q)d\sigma (Q) \nonumber \\
		& = \sum_{j=1}^{\infty} \int_{2^j\dt\setminus 2^{j-1}\dt } |f-f_{\dt}| K(A,Q)\bar{k}(Q)d\sigma (Q). 
	\end{align}
	
	Let $\Delta'$ be any surface ball contained in $2^j\dt\setminus 2^{j-1}\dt$, and $A_j$ be a corkscrew point in $T(2^j\dt)$. Then by Corollary 1.3.8 \cite{CBMS} (It follows easily from the boundary comparison principle \eqref{P6})
	\begin{equation}
		\dfrac{\omega(\Delta')}{\omega(2^j\dt)} \approx \omega^{A_j}(\Delta'). \label{condproba}
	\end{equation} 
	On the other hand, by the boundary regularity of $\omega^{X}(\Delta')$ in $T(2^{j-1} \dt)$ , we have
	\begin{align}
		\omega^A(\Delta') & \leq C\left(\dfrac{|A-Q_0|}{2^{j-1} \cdot c r}\right)^{\beta} \sup_{Y\in 2^{j-1}\dt} \omega^Y(\Delta') \nonumber \\
		& \leq C\left(\dfrac{9r}{2^{j-1} c r}\right)^{\beta} \omega^{A_{j}}(\Delta') \nonumber \\
		& \lesssim 2^{-j\beta} \omega^{A_j}(\Delta'). \label{Holderforw}
	\end{align}
	Combining \eqref{condproba} and \eqref{Holderforw} we have
	\begin{align*}
		\dfrac{\omega^A(\Delta')}{\omega(\Delta')} & \lesssim   2^{-j\beta} \dfrac{ \omega^{A_j}(\Delta')}{\omega(\Delta')} \approx \dfrac{2^{-j\beta}}{\omega(2^j \dt)} ,\\
	\end{align*}
	for any $\Delta'$ contained in $2^j \dt\setminus 2^{j-1}\dt$.
	Therefore by the definition \eqref{def:kernalK}
	\begin{equation}\label{eq:Kest}
		\sup_{Q\in 2^j\dt\setminus 2^{j-1}\dt} K(A,Q) \lesssim \dfrac{2^{-j\beta}}{\omega(2^j \dt)}.
	\end{equation} 
	Combining \eqref{eq:Bq}, \eqref{eq:vexp} and \eqref{eq:Kest},
	\begin{align}\label{eq:vpri}
		v(A) & \lesssim \sum_{j} \dfrac{2^{-j\beta}}{\omega(2^j \dt)} \left(\int_{2^j\dt} |f-f_{\dt}|^p d\sigma\right)^{1/p} \left(\int_{2^j\dt} \bar{k}^q d\sigma\right)^{1/q} \nonumber \\
		& \lesssim \sum_{j} 2^{-j\beta} \left(\fint_{2^j\dt} |f-f_{\dt}|^p d\sigma\right)^{1/p} \nonumber \\
		& \lesssim \|f\|_{BMO(\sigma)}.
	\end{align}
	Therefore 
	\begin{equation}\label{eq:vBMObd}
		v(X) \leq C \|f\|_{BMO(\sigma)} \qquad \text{for all~} X \in T(9\Delta).
	\end{equation}  
	
	For any $X\in T(3\Delta/2)$, let $Q_X$ be a boundary point such that $|X-Q_X|=\delta(X)$. Note that
	\[|X-Q_X|=\delta(X) \leq |X-Q_0| < \frac{3r}{2} , \]
	so $X\in B(Q_X,3r/2)\cap\Omega$. We consider the Dirichlet problem in $B(Q_X, 6r)\cap \Omega$. Note that
	\[ |Q_X-Q_0| \leq |Q_X-X|+|X-Q_0| <\frac{3r}{2} + \frac{3r}{2} =3r, \]
	hence $\overline{B\left(Q_X,6r\right)} \subset B(Q_0, 9r)$.
	Note that $\dt \supset 9\Delta \supset \Delta(Q_X, 6r)$, $v$ is a non-negative solution in $B(Q_X, 6r)\cap\Omega$ and vanishes on $\Delta(Q_X, 6r)$. By the boundary H\"older regularity (Proposition \ref{prop:bHr}) and the first part of this lemma \eqref{eq:vBMObd}, we conclude
	\begin{align*}
		v(X) \lesssim \left( \frac{|X-Q_X|}{3r/2}\right)^{\beta} \sup_{\overline{B(Q_X,6r)\cap\Omega}} v  \lesssim \left( \frac{\delta(X)}{r}\right)^{\beta} \sup_{T(9\Delta)} v  \lesssim \left( \frac{\delta(X)}{r}\right)^{\beta} \|f\|_{BMO(\sigma)} .
	\end{align*}
\end{proof}

Using Lemma \ref{lm:vBMObd} and \eqref{eq:utwo}, we get
\begin{equation}
	\iint_{T(\Delta)} \Carl{u_2} dX \lesssim \dfrac{\|f\|_{BMO(\sigma)}^2}{r^{2\beta}}\left( \iint_{T(\frac{3}{2}\Delta)} \delta(X)^{2\beta - 1} dX \right) . \label{eq:utwoBMO}
\end{equation} 
Note that $2\beta-1>-1$, we may use the following lemma.


\begin{lemma}\label{lm:delta} 
For any $\alpha>-1$, we have
	\begin{equation}
		\iint_{T(2\Delta)} \delta(X)^{\alpha} dX \lesssim r^{n+\alpha}. \label{eq:deltaintegral}
	\end{equation} 
\end{lemma}
\begin{proof}
	If $\alpha \geq 0$, the proof is trivial. For $j= 0,1,\cdots$ let 
	\[ T_j = T(2\Delta) \cap \{x\in\Omega: 2^{-j}r \leq \delta(X) < 2^{-j+1}r\}, \]
	\[ T_{<j} = T(2\Delta) \cap \{x\in\Omega: \delta(X) < 2^{-j+1}r\}. \]
	Then
	\begin{equation}
		\iint_{T(2\Delta)} \delta(X)^{\alpha} dX = \sum_{j=0}^{\infty} \iint_{T_j} \delta(X)^\alpha dX \lesssim \sum_{j=0}^{\infty} (2^{-j}r)^{\alpha} m(T_{<j}). \label{eq:intdelta}
	\end{equation}

	Consider a covering of $4\Delta$ by $4\Delta \subset \bigcup\limits_{Q\in 4\Delta} B(Q, 2^{-j+1}r)$, from which one can extract a countable Vitali sub-covering $4\Delta \subset \cup_{k} B(Q_k, 2^{-j+1}r)$,	where $Q_k\in 4\Delta$ and the balls $B_k = B(Q_k, 2^{-j+1}r/5)$ are pairwise disjoint. The fact that $Q_k\in 4\Delta = \Delta(Q_0, 4r)$ implies 
	\[ B_k = B\left(Q_k, \frac{2^{-j+1}r}{5}\right) \subset B\left(Q_0,4r + \frac{2^{-j+1}r}{5} \right).  \]
	And the pairwise disjointness of $B_k$'s implies there are only finitely many of them. In fact,
	\begin{equation}
		\sum\limits_{k} \sigma(B_k) = \sigma\left( \bigcup\limits_{k} B_k \right) \leq \sigma\left( \Delta\left(Q_0, 4r+ \frac{2^{-j+1}r}{5}\right)\right) \lesssim \left( 4r + \frac{2r}{5}\right)^{n-1}. \label{eq:disjoint}
	\end{equation} 
	Note that $\sigma(B_k) \approx \left( 2^{-j+1}r/5 \right)^{n-1}$ independent of $k$.
	Let $N$ be the number of $B_k$'s. By \eqref{eq:disjoint}
	\begin{equation}
		N\cdot \left( \frac{2^{-j+1}r}{5} \right)^{n-1} \leq \left( 4r + \frac{2r}{5}\right)^{n-1} , \qquad \text{thus } N\lesssim 2^{j(n-1)}. \label{eq:boundN}
	\end{equation}
	
	For any $X\in T_{<j}$, let $Q_X\in\pO$ be such that $|X-Q_X|=\delta(X)$. Then
	\begin{equation} \label{eq:closetoX}
		|Q_X-Q_0| \leq |Q_X-X| + |X - Q_0| < 4r, \qquad \text{i.e. } Q_X\in 4\Delta.
	\end{equation}
	Thus $Q_X \in B(Q_k, 2^{-j+1}r)$ for some $k$. Moreover
	$T_{<j} \subset \bigcup\limits_{k} B(Q_k, 2\cdot 2^{-j+1}r)$.
	Therefore
	\[ m(T_{<j}) \leq N\cdot \sup_{k} m(B(Q_k, 2\cdot 2^{-j+1}r)) \lesssim 2^{-j}r^n. \]
	Combined with \eqref{eq:intdelta} we get
	\[		\iint_{T(2\Delta)} \delta(X)^{\alpha} dX \lesssim \sum_{j=0}^{\infty} (2^{-j}r)^{\alpha}\cdot 2^{-j}r^n = r^{n+\alpha} \sum_{j=0}^{\infty} 2^{-j(\alpha+1)} \lesssim r^{n+\alpha}. \]
	The last sum is convergent because $\alpha+1 >0$.	
\end{proof}
Combining \eqref{eq:utwoBMO} and \eqref{eq:deltaintegral}, we get
\begin{equation}
	\iint_{T(\Delta)} \Carl{u_2} dX \lesssim r^{n-1} \|f\|_{BMO(\sigma)}^2 \lesssim \sigma(\Delta)\|f\|_{BMO(\sigma)}^2. \label{utwo}
\end{equation} 
\eqref{uone} and \eqref{utwo} together give the Carleson measure estimate \eqref{eq:CarlesonBMO}.

\section{From the Carleson measure estimate to $\omega_L \in A_{\infty}(\sigma)$}\label{sect:showAinfty}

Let $\Delta$ be an arbitrary surface ball. Let $f$ be a continuous non-negative function supported in $\Delta$ and $u$ is the solution with boundary value $f$. In particular $u$ is non-negative. Consider another surface ball $\Delta'$ of radius $r$ that is $r-$distance away from $\Delta$.
Then by assumption,
\begin{equation}
	\iint_{T(\Delta')} \Carl{u} dX\leq C\sigma(\Delta')\|f\|_{BMO(\sigma)}^2 \label{eq:Carlcond}
\end{equation}

We have shown in \eqref{eq:CarlesonSquareL} that
\begin{equation} \label{eq:CarlSquaregeq}
	\iint_{T(\Delta')}\Carl{u} dX \gtrsim\int_{\Delta'/2} S_{r/2}^2(u) d\sigma.
\end{equation}
In order to get a lower bound of the square function $S_{r/2}(u)$, we need to decompose the non-tangential cone $\Gamma_{r/2}(Q)$ as follows. 

\begin{center}
	\includegraphics{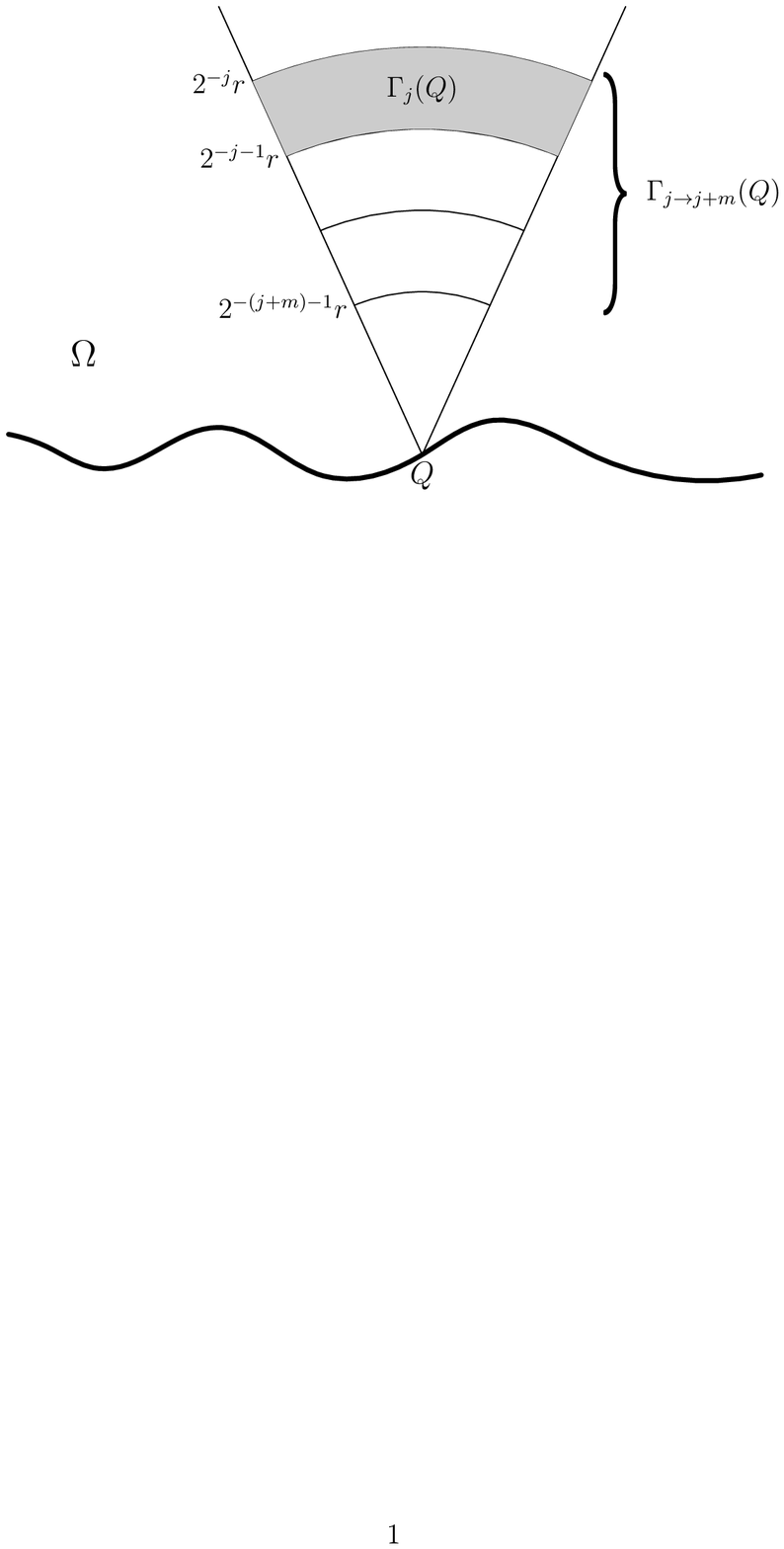}
\end{center}

\noindent For any $Q\in\Delta'/2$ and any $j\in\mathbb{N} $, let 
\begin{equation}
	\Gamma_j(Q) = \Gamma(Q) \cap \left(B_{2^{-j}r}(Q)\setminus B_{2^{-j-1}r}(Q)\right) \label{def:stripe}
\end{equation}  
be a stripe in the cone $\Gamma_{r/2}(Q)$ at height $2^{-j}r$, and 
\begin{equation}
	\Gamma_{j\rightarrow j+m}(Q) = \bigcup_{i=j}^{j+m}\Gamma_i(Q), \label{def:unionstripe}
\end{equation} 
a union of $(m+1)$ stripes. The above figure illustrates these notations, even though it over-simplifies the shape of the non-tangential cone $\Gamma(Q)$ and the relations between different radii. 

We claim the following Poincar\'e type inequality holds, because $u$ vanishes on $\Delta'$:
\begin{lemma}\label{lm:Poincare}
There exist an aperture $\overline\alpha>\alpha$, and integers $m_1, m_2$, such that the following Poincar\'e inequality holds for all $Q\in\Delta'/2$, 
	\begin{equation}
	\iint_{\Gamma_j^{\alpha}(Q)} u^2 dX \leq C (2^{-j}r)^2 \iint_{\Gamma_{j-m_1 \rightarrow j+m_2}^{\overline{\alpha} }(Q)} |\nabla u|^2 dX.\label{Poincareineq}
\end{equation}
The constants $m_1, m_2$ and $C$ only depend on $n$ and $\alpha$.
\end{lemma}

\subsection{Proof of Lemma \ref{lm:Poincare}: Poincar\'e inequality} \label{sect:Poincare}

The following lemma is the standard Poincar\'e inequality (see (7.45) and Lemma 7.16 in \cite{GlTr}).
\begin{lemma}
	Let $B$ be a ball in $\mathbb{R}^n$. If the function $u\in W^{1,2}(B)$, then
	\begin{equation}\label{eq:standardPoincare}
		\|u-u_{B}\|_{L^2(B)} \leq \left(\frac{\omega_n}{|B|} \right)^{1-1/n} \diam(B)^{n} \|\nabla u\|_{L^2(B)}.
	\end{equation}
	Here $u_{B} = \fint_{B} u dx$, $|B|$ is the $n-$dimensional Lebesgue measure of $B$ and $\omega_n$ is the volume of the unit ball in $\mathbb{R}^n$.
\end{lemma}
\noindent Let $r_B$ denote the radius of $B$, then we can rewrite \eqref{eq:standardPoincare} as
\begin{align}
	\iint_B |u(x) - u_B|^2 dx  \leq 4^n r_B^2 \iint_{B} |\nabla u(y)|^2 dy.\label{eq:PoincareinBall}
\end{align}

Assume there is a Harnack chain $B=B_1, B_2, \cdots, B_M = B'$ from $B\subset \Gamma_j^{\alpha}(Q)$ to $B'\subset\Gamma_{j+m}^{\alpha}(Q)$, by the triangle inequality
\begin{align}\label{eq:trg}
	\iint_{B} |u(x)-u_{B'}|^2 dx \leq & 2 \iint_{B} |u(x)-u_{B}|^2 dx + 
	 2M|B| \cdot \sum_{j=1}^{M-1} |u_{B_j}-u_{B_{j+1}}|^2.
\end{align}	 
Assume in addition that consecutive balls $B_i=B(x_i, r_i)$ and $B_{i+1}=B(x_{i+1}, r_{i+1})$ have comparable sizes
\begin{equation}
	c r_{i+1} \leq r_{i} \leq C r_{i+1}, \label{eq:radratio}
\end{equation} 
with constants $0<c<C$.
 We want to estimate $|u_{B_{i}}-u_{B_{i+1}}|$ by the integral of $\nabla u$. By $B_i \cap B_{i+1} \neq \emptyset$ and \eqref{eq:radratio}, we know 
 \begin{equation}
 	B_{i+1}\subset \lambda B_{i} \quad \text{with~} \lambda = 1+2/c>1. \label{eq:radratioCCL}
 \end{equation} 
Hence
\begin{align}
	|u_{B_{i+1}} -u_{\lambda B_{i}}|^2 
	& \leq \left( \frac{1}{|B_{i+1}|} \iint_{B_{i+1}}|u(x)-u_{\lambda B_i}| dx \right)^2 \nonumber \\
	& \leq \left(\frac{\lambda r_i}{r_{i+1}}\right)^{2n}  \fint_{\lambda B_i} |u(x) - u_{\lambda B_i}|^2 dx \nonumber \\
	& \lesssim \dfrac{\lambda^{n+2} C^{2n}}{|B_i|} r_i^2 \iint_{\lambda B_i} |\nabla u(y)|^2 dy \qquad \text{by \eqref{eq:PoincareinBall}}.
\end{align}
Similarly
\[ |u_{B_i}-u_{\lambda B_i}|^2 \lesssim \frac{\lambda^{n+2}}{|B_i|}r_i^2 \iint_{\lambda B_i}|\nabla u(y)|^2 dy. \]
Therefore
\begin{equation}
	|u_{B_i}-u_{B_{i+1}}|^2 \leq \frac{A(n, c, C)}{|B_i|}r_i^2\iint_{\lambda B_i}|\nabla u(y)|^2 dy. \label{eq:compave}
\end{equation} 
Plugging \eqref{eq:compave} back into \eqref{eq:trg}, we get
\begin{align}
	\iint_{B} |u(x)-u_{B'}|^2 dx & \leq C(n,M,c, C) r_B^2 \sum_{k=1}^{M-1} \iint_{\lambda B_k} |\nabla u(y)|^2 dy \nonumber \\
	& \leq \widetilde{C}(n,M,c, C) r_B^2  \iint_{\bigcup\limits_{k=1}^{M-1} \lambda B_k} |\nabla u(y)|^2 dy . \label{eq:uuave}
\end{align} 

On the other hand, by assumption the last ball $B'\subset \Gamma_{j+m}^\alpha(Q)$, we have 
\begin{align}
	u_{B'}^2 \leq \sup_{\Gamma^{\alpha}_{j+m}(Q)} u^2  \leq C\left(\frac{2^{-(j+m)}r}{2^{-j}r}\right)^{2\beta} \sup_{B(Q,2^{-j}r)\cap \Omega} u^2  \lesssim_{n,\alpha} 2^{-2\beta m} u^2(A_j). \label{eq:lastballcork}
\end{align}
where $A_j$ is the corkscrew point in $B(Q,2^{-j}r)\cap\Omega $. The second inequality is by the boundary regularity \eqref{P1} and the fact that $u$ vanishes on $\Delta'$, and the last inequality by \eqref{P3}. Since $B\subset \Gamma_j^\alpha(Q)$ is a non-tangential ball and $u$ is non-negative, by the Harnack principle
\[ u(x) \geq c_0 u(A_j) \quad\text{for all~} x\in B, \]
for a constant $c_0<1$.
Hence 
\begin{equation}
	u_{B'}^2 \lesssim 2^{-2\beta m}u(A_j)^2 \lesssim 2^{-2\beta m} \fint_{B} u^2 dx. \label{eq:estlastball}
\end{equation}

Combining \eqref{eq:uuave} and \eqref{eq:estlastball}, we obtain
\begin{align}
	\iint_B u^2 dx & \leq 2  |B|\left(u_{B'}\right)^2 + 2\iint_B |u(x) - u_{B'}|^2 dx \nonumber \\
	& \leq A(n,\alpha,c,C) 2^{-2\beta m} \iint_{B} u^2 dx + \widetilde{C}(n,M,c, C) r_B^2  \iint_{\bigcup\limits_{k=1}^{M-1} \lambda B_k} |\nabla u(y)|^2 dy. \label{eq:beforeabsorb}
\end{align}
Choose $m$ big enough such that $A(n,\alpha, c, C) 2^{-\beta m} \leq 1/2$, then we can absorb the first term on the right hand side of \eqref{eq:beforeabsorb} to the left, and obtain
\begin{equation}
	\iint_B u^2 dx \lesssim_{n,\alpha,c,C} r_B^2 \iint_{\bigcup\limits_{k=1}^{M-1} \lambda B_k} |\nabla u(y)|^2 dy. \label{eq:Poincaresketch}
\end{equation}
Note that after $m$ is fixed, by the Harnack chain condition (see Definition \ref{def:HCC}), the number of balls $M\leq C(m)$ is also fixed, thus we omit the dependence on $M$ in the above inequality. This is Poincar\'e inequality for non-tangential balls.

In order to prove the Poincar\'e inequality \eqref{Poincareineq} for non-tangential cones, we just need to cover $\Gamma_j^{\alpha}(Q)$ by balls; we also need to choose the Harnack chain carefully so that the integration region $\cup_{k=1}^{M-1} 2\lambda B_k$ in the right hand side of \eqref{eq:Poincaresketch} is also contained in a non-tangential cone, possibly of a bigger aperture and wider stripe.
Let us first make the following simple observation:
\begin{obs}\label{obs:one}
	Let $B$ be a Harnack ball with constants $(C_1,C_2)$ (see the definition of Harnack balls in \eqref{Harnackball}). Assume $B$ contains some point $X\in\Gamma^{\alpha}(Q)$, 
	then  
	\[ B\subset \Gamma^{\alpha(1+\widetilde{C}_1)}(Q), \quad \widetilde C_1>0 \text{~ is a constant only depending on~} C_1. \] 
\end{obs}

If in adition $X\in\Gamma^{\alpha}_j(Q)$, then $|X-Q| \approx 2^{-j}r$, and we can get more precise estimate:
\begin{obs}\label{obs:two}
	Assume $B$ contains a non-tangential point $X\in\Gamma_j^{\alpha}(Q)$, then
\[ B\subset \mathlarger{\Gamma}^{\alpha(1+\widetilde C_1)}_{j-1\rightarrow j+n_0}(Q),\quad n_0 \text{~is an integer depending only on~} \alpha \text{~and~} C_1 . \]
\end{obs}

Moreover, by induction:
\begin{obs}\label{obs:four}
	 If $B_1, B_2, \ldots, B_k$ are Harnack balls with constants $(C_1, C_2)$ such that $B_j \cap B_{j+1} \neq \emptyset$, and $B_1$ contains some point $X\in \Gamma_j^{\alpha}(Q)$, then
	 \[ \bigcup\limits_{j=1}^k B_j \subset \mathlarger{\Gamma}^{\alpha(1+\widetilde C_1)^k}_{j-k\rightarrow j+n_0 + (k-1)m_0}(Q),\]
	 where $m_0$ is an integer depending only on $C_1$ (more precisely, $m_0$ is such that $1/(1+\widetilde C_1) \geq 2^{-m_0}$).
\end{obs}

Let $A$ and $A'$ be arbitrary points in $\Gamma_j^{\alpha}(Q)$ and $\Gamma_{j+m}^{\alpha}(Q)$ respectively. Then 
\[ \rho = \min(\delta(A), \delta(A')) \geq 2^{-(j+m)-1}r/\alpha , \qquad |A-A'| \leq 2\cdot 2^{-j}r \lesssim 2^{m}\rho . \]
By the Harnack chain condition, there is a chain of open Harnack balls $B_1, B_2, \cdots, B_M$ with constants $(C_1, C_2)$ 
that connects $A$ to $A'$, and the number of balls $M\leq C(m)$. By Observation \ref{obs:two}, the balls $B_1$ and $B_M$ are non-tangential balls of aperture $\alpha(1+\widetilde C_1)$:
\[B_1 \subset \mathlarger{\Gamma}^{\alpha(1+\widetilde C_1)}_{j-1\rightarrow j+n_0}(Q), \quad B_M \subset \mathlarger{\Gamma}^{\alpha(1+\widetilde C_1)}_{j+m-1\rightarrow j+m+n_0}(Q). \]
Simple computation shows that the sizes of two consecutive Harnack balls are comparable:
\begin{equation}
	\frac{C_1}{C_2+1} \diam(B_j) \leq \diam(B_{j+1}) \leq \frac{C_2+1}{C_1} \diam(B_j). \label{eq:radratioHB}
\end{equation} 
Recall we showed in Section \ref{sect:Poincare} (see \eqref{eq:radratioCCL}) that \eqref{eq:radratioHB} implies $B_{j+1} \subset \lambda B_j$, with a constant $\lambda$ depending on $C_1, C_2$. 
As we have discussed in Remark (2) after Definition \ref{def:HCC}, if $\Omega$ satisfies the Harnack chain condition with \eqref{preHarnackball}, we may choose $C_1$ (lower bound constant for the Harnack ball) appropriately such that the enlarged ball $\widetilde B_j = \lambda B_j$ still lies in $\Omega$, and moreover, its distance to the boundary is still comparable to its diameter. More precisely, it is an easy exercise to show that if we choose $C_1 \approx 26C^4$, the enlarged balls $\widetilde B_j$'s are still Harnack balls with modified constants:
\begin{equation}
	\frac{3}{2} \diam(\widetilde B_j) \leq \delta(\widetilde B_j) \leq C_2 \diam(\widetilde B_j). \label{HBforLB}
\end{equation}

Denote $\IT(A, A') = \cup_{i=1}^{M-1} \widetilde B_i$ (IT stands for \textquotedblleft integration tube\textquotedblright). By \eqref{HBforLB} and Observation \ref{obs:four},  
\[ \IT(A,A')\subset \mathlarger{\Gamma}^{h(\alpha, m)}_{j-M\rightarrow j + n_0 + M m_0 }(Q), \]
where $h(\alpha, m), n_0, m_0$ depend on the constants of the Harnack balls $3/2, C_2$, the number of balls $M = O(m)$ and the aperture $\alpha$ we start with. Thus by \eqref{eq:Poincaresketch}
\begin{align*}
	\iint_{B_1} u^2(x) dx \lesssim_{n,\alpha} r_{B_1}^2 \iint_{\IT(A,A')} |\nabla u(y)|^2 dy \lesssim_{n,\alpha} (2^{-j}r)^2 \iint_{\mathlarger{\Gamma}^{h(\alpha, m)}_{j-M\rightarrow j + n_0 + M m_0 }(Q)} |\nabla u(y)|^2 dy.
\end{align*}

To summarize, for any $A\in \Gamma_j^{\alpha}(Q)$, we can find a Harnack ball $B$ containing $A$ which satisfies $B\subset \Gamma^{\alpha(1+\widetilde C_1)}_{j-1\rightarrow j+n_0}(Q) $ and
\begin{align}
	\iint_{B} u^2(x) dx \lesssim_{n,\alpha} (2^{-j}r)^2 \iint_{\mathlarger{\Gamma}^{\overline \alpha }_{j-m_1\rightarrow j + m_2 }(Q)} |\nabla u(y)|^2 dy , \label{eq:PoincareNTball}
\end{align}
where the aperture $\overline \alpha> \alpha$, and $n_0, m_1, m_2$ are integers depending on $\alpha$. We cover $\Gamma_j^{\alpha}(Q)$ by such Harnack balls:
\begin{equation}
	\Gamma_j^{\alpha}(Q) \subset \bigcup\limits_{X\in \Gamma_j^{\alpha}(Q)} B^X \subset \Gamma^{\alpha(1+\widetilde C_1)}_{j-1\rightarrow j+n_0}(Q), \label{covering}
\end{equation} 
from which we can extract a Vitali sub-covering $\Gamma_j^{\alpha}(Q) \subset \cup_k B^k$, such that $\{B^k/5\}$ are pairwise disjoint.
By the definition \eqref{def:stripe} and \eqref{def:unionstripe}, the set $\Gamma^{\alpha(1+\widetilde C_1)}_{j-1\rightarrow j+n_0}(Q)$ is contained in an annulus with small radius $2^{-(j+n_0)-1}r$ and big radius $2^{-(j-1)}r$. By the disjointedness of $\{B^k/5\}$'s and the fact that each $B^k/5$ has radius comparable to $2^{-j}r$, we can show that the number of balls in the Vitali covering is bounded by a constant $N=N(n, \alpha)$. 

 Finally, by the finite overlap of Vitali covering and \eqref{eq:PoincareNTball}, we have
\begin{align*}
	\iint_{\Gamma_j^{\alpha}(Q)} u^2(x) dx & \leq C(n) \sum_{k} \iint_{B^k} u^2 dX \\
	& \lesssim  N(n,\alpha) \cdot (2^{-j}r)^2 \iint_{\Gamma^{\overline\alpha }_{j-m_1\rightarrow j + m_2 }(Q)} |\nabla u(y)|^2 dy .
\end{align*}
This finishes the proof of Lemma \ref{lm:Poincare}.

\subsection{From the Carleson measure estimate to estimate of the boundary value}

Given $\alpha>0$, let $\overline\alpha> \alpha$ and $m_1, m_2 \in \mathbb{N}$ be defined as in Lemma \ref{lm:Poincare}. Using the Poincar\'e type inequality \eqref{Poincareineq}, we can get a lower bound of the square function (defined in the cone $\Gamma^{\overline\alpha}(Q)$):
\begin{align*}
	|S_{r/2}^{\overline\alpha }(u)(Q)|^2 & = \iint_{\Gamma^{\overline \alpha}_{r/2} (Q)} |\nabla u|^2 \delta(X)^{2-n} dX \\
	& \geq \frac{1}{m_1+m_2} \sum_{j=m_1+1}^{\infty} \iint_{\Gamma^{\overline\alpha}_{j-m_1 \rightarrow j + m_2}(Q)} |\nabla u|^2 \delta(X)^{2-n} dX \\
	& \gtrsim \sum_{j=m_1}^{\infty}(2^{-j}r)^{2-n} \iint_{\Gamma^{\overline \alpha}_{j-m_1\rightarrow j+m_2 } (Q)} |\nabla u|^2 dX \\
	& \gtrsim \sum_{j=m_1}^{\infty} (2^{-j}r)^{2-n}\cdot (2^{-j}r)^{-2} \iint_{\Gamma_j^{\alpha}(Q)} u^2 dX \quad\quad \text{by~} \eqref{Poincareineq} \\
	& \gtrsim \sum_{j=m_1}^{\infty} u^2(A_j),
\end{align*}
where $A_j\in \Gamma_j(Q)$ is a corkscrew point at the scale $2^{-j}r$.
In the last inequality, we use the interior corkscrew condition, thus each stripe of cone $\Gamma_j(Q)$ contains a ball of radius comparable to $2^{-j-1}r$ (as long as $\alpha$ is chosen to be big, say $\alpha>2M$, where $M$ is the corkscrew constant). By the Harnack principle $u(A_{j+1}) \geq c u(A_j)$, where $c<1$ is a constant independent of $u$ and $j$. Thus
\[  \sum_{j=m_1}^{\infty} u^2(A_j) \gtrsim u^2(A_{m_1}) \gtrsim u^2(A_1). \]
Recall for any $Q\in \Delta'/2$, the point $A_1= A_1(Q)$ is a corkscrew point in $\Gamma_1(Q)$. Let $A'$ be the corkscrew point in $T(\Delta'/2)$, again by the Harnack principle we have $u(A') \approx u(A_1)$. 
Therefore
\[ |S_{r/2}^{\overline\alpha }(u)(Q)|^2 \gtrsim u^2(A_1) \gtrsim u^2(A'), \quad \text{for any~} Q\in \Delta'/2 . \]
Combining this with \eqref{eq:Carlcond} and \eqref{eq:CarlSquaregeq}, we get
\[ \sigma(\Delta')\|f\|_{BMO(\sigma)}^2 \gtrsim \int_{\Delta'/2} |S_{r/2}^{\overline\alpha }(u)|^2 d\sigma \gtrsim \sigma(\Delta'/2)u^2(A') \gtrsim \sigma(\Delta')u^2(A') , \]
and thus
\[ u(A') \lesssim \|f\|_{BMO(\sigma)}. \]

Let $A$ be a corkscrew point in $\Delta$. Since $\Delta$ and $\Delta'$ have the same radius $r$ and they are $r-$distance apart, we have $u(A) \approx u(A')$. By the assumption $f$ is supported on $\Delta$,
\[u(A) = \int_{\Delta} f(Q) d\omega^A(Q) = \int_{\Delta} f(Q) K(A,Q) d\omega(Q) \approx \frac{1}{\omega(\Delta)} \int_{\Delta}f d\omega. \]
The last equality uses the estimate of $K(A,Q)$ when $A$ is a corkscrew point in $T(\Delta)$ and $Q\in \Delta $ (see \cite{CBMS} Corollary 1.3.8). As a result, we proved the following estimate: Let $f$ be a non-negative continuous function supported on $\Delta$, then
\begin{equation}\label{eq:BMO}
	\frac{1}{\omega(\Delta)} \int_{\Delta}f d\omega \leq C \|f\|_{BMO(\sigma)}.
\end{equation} 

\subsection{Proof of $\omega_L\in A_{\infty}(\sigma)$}
Let $\Delta$ be a surface ball with radius $r$. For $\epsilon>0$ fixed, we want to find an $\eta= \eta(\epsilon) $, such that for any $E \subset \Delta$,
\[ \frac{\sigma(E)}{\sigma(\Delta)} <\eta \quad\text{implies} \quad \frac{\omega (E)}{\omega (\Delta)} < \epsilon. \]
In fact, since $\sigma$ and $\omega$ are Borel measures, we may assume $E$ is an open subset of $\Delta$.

Let $\delta>0$ be a small constant to be determined later, we define the function
\begin{equation}
	f(x) = \max\left\{ 0, 1+\delta \log M_{\sigma}\chi_{E}(x)\right\} \label{def:f}
\end{equation} 
where $M_{\sigma}$ is the Hardy-Littlewood maximal function with respect to $\sigma$:
\begin{equation}
	M_{\sigma}\chi_E(x) = \sup_{\widetilde \Delta \ni x} \dfrac{\sigma(\widetilde\Delta\cap E)}{\sigma(\widetilde\Delta)}. \label{def:M}
\end{equation} 
Since $(\pO, \sigma)$ is a space of homogeneous type, we can adapt the arguments in \cite{logBMO} very easily and show that $\|\log M_{\sigma}\chi_{E}\|_{BMO(\sigma)}$ is bounded by some constant $A$ (independent of the set $E$). Hence $f$ is a BMO function and $\|f\|_{BMO(\sigma)} \leq A\delta$. 
Moreover, it is clear from the definitions \eqref{def:f} and \eqref{def:M} that $0\leq f\leq 1$ and $f \equiv 1$ on the open set $E$.

Suppose $x\in\pO \setminus 2\Delta $, then $\dist(x,\Delta) \geq r$. Let $\dt$ be an arbitrary surface ball containing $x$.
Since $E\subset \Delta$, in order for $\widetilde\Delta\cap E$ to be nonempty, the diameter of $\widetilde\Delta$ is at least $r$. Thus by Ahlfors regularity $\sigma(\widetilde\Delta)\gtrsim r^{n-1}\approx \sigma(\Delta)$.
Therefore
\begin{equation}
	M_{\sigma}\chi_E(x) = \sup_{\widetilde \Delta \ni x} \dfrac{\sigma(\widetilde\Delta\cap E)}{\sigma(\widetilde\Delta)} \leq C\dfrac{\sigma(E)}{\sigma(\Delta)}. \label{est:Mchi}
\end{equation} 
This means, as long as $E\subset \Delta$ is such that 
\begin{equation}
	\frac{\sigma(E)}{\sigma(\Delta)}  <\eta(\delta) = \frac{e^{-1/\delta}}{C}, \label{conditioneta}
\end{equation} 
by \eqref{est:Mchi} we have
\[ 1+\delta \log M_{\sigma}\chi_E(x) < 1+\delta \log e^{-1/\delta} =0 \text{~outside of~} 2\Delta,\]
hence $f\equiv 0$ outside of $2\Delta$. In other words, 
\[ \frac{\sigma(E)}{\sigma(\Delta)} <\eta \implies f \text{~is supported in~} 2\Delta. \]

Next we want to use a mollification argument to approximate $f$ by continuous functions, such that their BMO norms are uniformly bounded by that of $f$. Let $\varphi$ be a radial-symmetric smooth function on $\mathbb{R}^n$ such that $\varphi = 1$ on $B_{1/2}$, $\supp \varphi\subset B_1 $ and $0\leq \varphi \leq 1$. Let 
\begin{equation}
	\varphi_{\epsilon}(z) = \frac{1}{\epsilon^{n-1}} \varphi\left(\frac{z}{\epsilon}\right), \quad f_{\epsilon}(x) = \frac{ \int_{y\in \pO} f(y) \varphi_{\epsilon} (x-y) d\sigma(y)}{\int_{y\in\pO} \ve(x-y)d\sigma(y)} \text{ for } x\in\pO . \label{deffe}
\end{equation} 
The following lemma summarizes the properties of these $f_{\epsilon}$'s. The proof of (1) is just a standard mollification argument. However it requires more work to prove (2) and (3), since $f_{\epsilon}$ is a (normalized) convolution of $f$ restricted to $\pO$, instead of all of $\mathbb{R}^n$. In particular, the proof depends on the properties of such function $f$ defined in \eqref{def:f}. We refer interested readers to Appendix \ref{sect:feprop} for the proof.

\begin{lemma}\label{lm:fe}
	The following properties hold for $f_{\epsilon}$'s:
\begin{enumerate}
	\item each $f_{\epsilon}$ is continuous, and is supported in $3\Delta$;
	\item there is a constant $C$ (independent of $\epsilon$) such that $\|f_{\epsilon}\|_{BMO(\sigma)} \leq C \|f\|_{BMO(\sigma)}$;
	\item $f(x) \leq \liminf_{\epsilon\rightarrow 0} f_{\epsilon}(x)$ for all $x$ in their support $3\Delta$. 
\end{enumerate} 
\end{lemma}
The last property and Fatou's lemma imply
\begin{align}
	\int_{3\Delta} f(x) d\omega(x) \leq \int_{3\Delta} \liminf_{\epsilon\rightarrow 0} f_{\epsilon}(x) d\omega(x) \leq \liminf_{\epsilon\rightarrow 0} \int_{3\Delta} f_{\epsilon}(x) d\omega(x) . \label{eq:Fatou}
\end{align} 
Since each $f_{\epsilon}$ is continuous, we can apply \eqref{eq:BMO},
\begin{equation}
	\frac{1}{\omega(3\Delta)} \int_{3\Delta} f_{\epsilon}(x) d\omega(x) \leq C \|f_{\epsilon}\|_{BMO(\sigma)}\leq C' \|f\|_{BMO(\sigma)}. \label{eq:feBMO}
\end{equation} 
Combining \eqref{eq:Fatou} and \eqref{eq:feBMO}, we get
\[ \frac{1}{\omega(3\Delta)} \int_{3\Delta} f(x) d\omega(x)\leq C' \|f\|_{BMO(\sigma)} \leq C''\delta. \]
On the other hand, since $f\geq\chi_E$ and $\omega$ is a doubling measure,
\[ \frac{1}{\omega(3\Delta)} \int_{3\Delta} f(x) d\omega(x) \geq \frac{\omega(E)}{\omega(3\Delta)} \gtrsim \frac{\omega(E)}{\omega(\Delta)}. \]
 Therefore $\omega(E)/\omega(\Delta) \leq C\delta$ as long as the condition \eqref{conditioneta}, i.e. $\sigma(E)/\sigma(\Delta)<\eta$ is satisfied. In other words, $\omega \in A_{\infty}(\sigma)$.

\section{Converse to the Carleson measure estimate}\label{sect:converse}

\begin{prop}\label{prop:reverseCarl}
Assume the elliptic measure $\omega \in A_{\infty}(\sigma)$. If $Lu=0$ in $\Omega$ with boundary data $f\in C(\pO)$, then 
\begin{equation} \label{reverseCarl}
	\|f\|_{BMO(\sigma)}^2 \leq C \sup_{\Delta\subset\pO} \frac{1}{\sigma(\Delta)}\iint_{T(\Delta)}\Carl{u} dX,
\end{equation}
as long as the right hand side is finite.
\end{prop}

Since $\omega \in A_{\infty}(\sigma)$, a classical result in harmonic analysis says $\|f\|_{BMO(\sigma)} \approx \|f\|_{BMO(\omega)}$. We may as well prove \eqref{reverseCarl} for $\|f\|_{BMO(\omega)}$.
In light of previous work \cite{FN} and \cite{FKN}, Jerison and Kenig studied the Dirichlet problem with BMO boundary data and proved the following Theorem \ref{thm:BMOGreen} for \textit{the Laplacian on NTA domains} (see \cite{NTA} Theorem 9.6). The main ingredients of their proof are \eqref{P4}, \eqref{P5} and a geometric localization theorem, which we have shown to hold for \textit{general elliptic operators $L$ on uniform domains with Ahlfors regular boundary} in Section \ref{sect:BR}. Therefore a similar proof is applicable in our case. 

\begin{theorem}\label{thm:BMOGreen}
There exists a constant $C>0$ such that
	\begin{equation}\label{one}
		\|f\|_{BMO(\omega)}^2 \leq C \sup_{\Delta\subset\pO} \frac{1}{\omega(\Delta)}\iint_{T(\Delta)}|\nabla u|^2 G(X_0, X) dX,
	\end{equation}
	on condition that the right hand side is bounded.
\end{theorem}

\begin{proof}
	Let $\Delta=\Delta(Q,r)$ be an arbitrary surface ball. By Theorem \ref{thm:GLT} there is a uniform domain $\mathcal{D}$ with Ahlfors regular boundary satisfying $B(Q,4r)\cap\Omega \subset \mathcal{D} \subset B(Q,4Cr)\cap\Omega$.
	 Assume $r$ is small so that $X_0 \notin B(Q, 4Cr)$. By the interior corkscrew condition of $\mathcal{D}$, we can find a point $X_1 \in \mathcal{D}\setminus B(Q,2r)$ and $\delta_1(X_1) := \dist(X_1, \partial\mathcal{D}) \approx 2r$. For the elliptic operator $L$ on $\mathcal{D}$, let $\nu$ be the elliptic measure and $G_{\mathcal{D}}(X_1,\cdot)$ the Green's function with pole at $X_1$. 
	
	We claim that for any surface ball $\Delta'\subset \Delta$.
	\begin{equation}
		\frac{\omega(\Delta')}{\omega(\Delta)} \approx \nu(\Delta'). \label{cmpon}
	\end{equation}
	In fact, let $Y_0$ and $Y'$ be corkscrew points with respect to $\Delta$ and $\Delta'$ respectively, and let $r'$ be the radius of $\Delta'$. Apply \eqref{P4} to the domains $\Omega$ and $\mathcal{D}$, we get
	\begin{equation}
		\frac{\omega(\Delta')}{\nu(\Delta')} \approx \dfrac{G(X_0,Y')(r')^{n-2}}{G_{\mathcal{D}}(X_1, Y')(r')^{n-2}} \approx \dfrac{G(X_0,Y')}{G_{\mathcal{D}}(X_1,Y')}. \label{cmpp}
	\end{equation} 
	And similarly
	\begin{equation}
		\frac{\omega(\Delta)}{\nu(\Delta)} \approx \dfrac{G(X_0,Y_0)}{G_{\mathcal{D}}(X_1,Y_0)}. \label{cmp}
	\end{equation} 
	Note that $X_0, X_1\notin B(Q,2r)\cap\Omega $, by the boundary comparison principle \eqref{P6}
	\begin{equation}
		\dfrac{G(X_0,Y')}{G_{\mathcal{D}}(X_1,Y')} \approx \dfrac{G(X_0,Y_0)}{G_{\mathcal{D}}(X_1,Y_0)}. \label{cmpG}
	\end{equation} 
	It follows from \eqref{cmpp}, \eqref{cmp} and \eqref{cmpG} that $\omega(\Delta')/\nu(\Delta') \approx \omega(\Delta)/\nu(\Delta)$.
	Since $\nu(4C\Delta) = \nu(\partial\mathcal{D}) = 1 $ and $\nu$ is a doubling measure, we have $\nu(\Delta) \approx 1$ and thus \eqref{cmpon}.
	
	The above estimate \eqref{cmpon} in particular implies
	\begin{equation}
		\frac{1}{\omega(\Delta)} \int_{\Delta} |f-c_{\Delta}|^2 d\omega \approx \int_{\Delta} |f-c_{\Delta}|^2 d\nu \leq \int_{\partial\mathcal{D} } |u-c_{\Delta}|^2 d\nu. \label{localize}
	\end{equation} 
	Here we choose the constant $c_{\Delta}= \int_{\partial\mathcal{D}} u d\nu $, so that $\int_{\partial\mathcal{D}} \left(u-c_{\Delta}\right) d\nu = 0$. We have the following global estimate on $\mathcal{D}$ (see Lemma 1.5.1 in \cite{CBMS})
	\begin{align}
		\frac{1}{2} \int_{\partial\mathcal{D}} |u-c_{\Delta}|^2 d\nu & = \iint_{\mathcal{D}} A\nabla u\cdot \nabla u ~ G_{\mathcal{D}} (X_1,Y) dY \leq \lambda \iint_{\mathcal{D}} |\nabla u|^2 G_{\mathcal{D}} (X_1,Y) dY.\label{globalL2} 
	\end{align}
	
	Using \eqref{cmp} and $\nu(\Delta)\approx 1$, we have
	\begin{equation}
		G_{\mathcal{D}}(X_1,Y_0) \approx \dfrac{G(X_0,Y_0)}{\omega(\Delta)} \nu(\Delta) \approx \dfrac{G(X_0,Y_0)}{\omega(\Delta)}. \label{eq:cmpfixed}
	\end{equation} 
	Since $X_0\notin\mathcal{D}$, by considering Harnack chains in $\mathcal{D}\setminus B(X_1, \delta_1(X_1)/3)$ or using the boundary comparison principle \eqref{P6}, \eqref{eq:cmpfixed} implies
	\[ G_{\mathcal{D}}(X_1,Y) \approx \dfrac{G(X_0,Y)}{\omega(\Delta)},\quad \text{for all~} Y\in \mathcal{D}\setminus B\Big(X_1, \frac{\delta_1(X_1)}{3}\Big). \]
	Thus 
	\begin{equation}
		\iint_{\mathcal{D}\setminus B\big(X_1, \frac{\delta_1(X_1)}{3}\big)} |\nabla u|^2 G_{\mathcal{D}} (X_1,Y) dY \approx \frac{1}{\omega(\Delta)} \iint_{\mathcal{D}} |\nabla u|^2 G (X_0,Y) dY. \label{awayX1}
	\end{equation} 
	
	On the other hand on $B(X_1,\delta_1(X_1)/3)$, by the Harnack principle
	\begin{equation}
		|\nabla u|^2 \lesssim \frac{1}{\delta_1(X_1)^n} \iint_{B\big(X_1, \frac{\delta_1(X_1)}{3}\big)} |\nabla u|^2 dY. \label{ptbyavrg}
	\end{equation} 
	By the choice of $X_1$ we know $\delta(X_1) \approx \delta_1(X_1) \approx 2r$ and $X_1\in B(Q,4Cr)$. Let $Q_{X_1}\in\pO$ satisfy $|X_1-Q_{X_1}|=\delta(X_1)$, then by properties \eqref{P4} and \eqref{P5},
	\begin{equation}
		\delta_1(X_1)^{n-2} G(X_0,Y) \approx \omega(\Delta(Q_{X_1}, \delta_1(X_1)) \approx \omega(\Delta) \label{omegaGreen}
	\end{equation} 
	for any $Y\in B(X_1,\delta_1(X_1)/3)$. Plugging \eqref{omegaGreen} into \eqref{ptbyavrg}, we get
	\begin{align}
		|\nabla u|^2 & \lesssim \frac{1}{\delta_1(X_1)^n} \frac{\delta_1(X_1)^{n-2}}{\omega(\Delta)} \iint_{B\big(X_1, \frac{\delta_1(X_1)}{3}\big)} |\nabla u|^2 G(X_0,Y) dY \nonumber \\
		& \lesssim \frac{1}{r^2 \omega(\Delta)} \iint_{B\big(X_1, \frac{\delta_1(X_1)}{3}\big)} |\nabla u|^2 G(X_0,Y) dY.\label{eq:grad}
	\end{align}
	By the maximal principle and the bound on Green's function,
	\begin{align}
		G_{\mathcal{D}}(X_1,Y) \leq G(X_1,Y) \lesssim \frac{1}{|Y-X_1|^{n-2}}.\label{eq:Green}
	\end{align}
	The last inequality is independent of $\mathcal{D}$ and $X_1$. Combining \eqref{eq:grad} and \eqref{eq:Green}, we get
	\begin{align}
		& \iint_{B\big(X_1, \frac{\delta_1(X_1)}{3}\big)} |\nabla u|^2 G_{\mathcal{D}} (X_1,Y) dY \nonumber \\
		& \qquad \qquad \lesssim \left(\frac{1}{r^2\omega(\Delta)} \iint_{B(X_1,\frac{\delta_1(X_1)}{3})} |\nabla u|^2 G(X_0,Y) dY \right) \cdot \iint_{B\big(X_1, \frac{\delta_1(X_1)}{3}\big)} \frac{1}{|Y-X_1|^{n-2} }dY\nonumber \\
		& \qquad \qquad \lesssim \left(\frac{1}{r^2\omega(\Delta)} \iint_{B(X_1,\frac{\delta_1(X_1)}{3})} |\nabla u|^2 G(X_0,Y) dY \right) \cdot \delta_1(X_1)^2 \nonumber \\
		& \qquad \qquad \lesssim \frac{1}{\omega(\Delta)} \iint_{\mathcal{D} } |\nabla u|^2 G(X_0,Y) dY \label{nearX1}.
	\end{align}
	
	Summing up \eqref{awayX1} and \eqref{nearX1}, we get
	\begin{equation}
		\iint_{\mathcal{D}} |\nabla u|^2 G_{\mathcal{D}} (X_1,Y) dY \lesssim \frac{1}{\omega(\Delta)} \iint_{\mathcal{D} } |\nabla u|^2 G(X_0,Y) dY.
	\end{equation}
	Together with \eqref{localize} and \eqref{globalL2}, we deduce	
	\begin{align*}
		\frac{1}{\omega(\Delta)} \int_{\Delta} |f-c_{\Delta}|^2 d\omega & \lesssim \frac{1}{\omega(\Delta)} \iint_{\mathcal{D}} |\nabla u|^2 G (X_0,Y) dY \\
		& \lesssim \frac{1}{\omega(4C\Delta)} \iint_{B(Q,4Cr)\cap\Omega} |\nabla u|^2 G (X_0,Y) dY \\
		& \lesssim \sup_{\Delta'\subset\pO} \frac{1}{\omega(\Delta')}\iint_{T(\Delta')}|\nabla u|^2 G(X_0, Y) dY.
	\end{align*}

\end{proof}

Recall for $X\in T(\Delta)$, the set $\Delta^X$ is defined as $\{Q\in \pO: X\in\Gamma(Q) \}$. By \eqref{DeltaX} and \eqref{P4}, \eqref{P5}, we get $G(X_0,X)\delta(X)^{n-2} \approx \omega(\Delta^X)$.
Thus by changing the order of integration,
\begin{align}
	\iint_{T(\Delta)}|\nabla u|^2 G(X_0, X) dX &\approx \iint_{T(\Delta)}|\nabla u|^2 \delta(X)^{2-n} \omega(\Delta^X) dX \nonumber \\
	& \leq \int_{Q\in (\alpha+1)\Delta} \iint_{X\in \Gamma_{\alpha r}(Q)} |\nabla u|^2 \delta(X)^{2-n} dX d\omega \nonumber \\
	& = \int_{\widetilde \Delta } S_{\alpha r}^2(u) d\omega, \label{two}
\end{align}
where $\widetilde \Delta = (\alpha+1)\Delta$. Since $\omega\in A_{\infty}(\sigma)$, there exists $q>1$ such that the Radon-Nikodym derivative $k=d\omega/d\sigma\in B_q(\sigma)$. Let $p>1$ be the conjugate of $q$, i.e. $1/p+1/q=1$, we have
\begin{align}
	\int_{\widetilde \Delta } S_{\alpha r}^2(u) d\omega  = \int_{\widetilde \Delta } S_{\alpha r}^2(u) k d\sigma & \leq   \left(\int_{\dt} k^q d\sigma\right)^{1/q} \left( \int_{\widetilde \Delta} S_{\alpha r}^{2p}(u) d\sigma \right)^{1/p} \nonumber \\
	& \leq C\sigma(\dt)^{1/q} \left(\fint_{\dt} k d\sigma\right) \left( \int_{\widetilde \Delta} S_{\alpha r}^{2p}(u) d\sigma \right)^{1/p} \nonumber \\
	& \leq C \omega(\Delta) \sup_{\Delta_{\tau}\subset\pO } \left(   \frac{1}{\sigma(\Delta_{\tau})} \int_{\Delta_{\tau}} S_{\tau}^{2p}(u) d\sigma \right)^{1/p}.  \label{three}
\end{align} 

We claim the following theorem holds for the truncated square function:
\begin{theorem}\label{tm:squareCarleson}
	For any $2<t<\infty$,
	\begin{equation} \label{four}
		\sup_{\substack{0<r<\diam\Omega \\ \Delta_{r} \subset\pO} } \left(\frac{1}{\sigma(\Delta_{r})} \int_{\Delta_r } \left|S_{r}(u)\right|^t d\sigma\right)^{1/t} \leq C  \sup_{\Delta\subset\partial\Omega} \frac{1}{\sigma(\Delta)} \iint_{T(\Delta)} \Carl{u} dX,
	\end{equation}
	on condition that the right hand side is finite. Here $\Delta_r$ denotes any surface ball of radius $r$.
\end{theorem}

Assume this theorem holds, then \eqref{reverseCarl} follows from combining \eqref{one}, \eqref{two}, \eqref{three} and \eqref{four}, which concludes the proof of Proposition \ref{prop:reverseCarl}. 
Now we are going to prove this theorem by combining several lemmas of the truncated square function. In \cite{HACAD}, the authors have proved similar lemmas for the square function (see Proposition 4.5, Lemma 4.6 and Lemma 6.2), and we are adapting their arguments to the \textit{truncated square function}. The proof of the following two lemmas is similar to the case of non-truncated square function, so we postpone it to Appendix \ref{sect:trsq}.

\begin{lemma} \label{lemma:open}
	For any $r,\lambda>0$, the set $\{Q\in\pO: S_r u(Q) >\lambda\}$ is open in $\pO$.
\end{lemma}
\begin{lemma} \label{lemma:diffapper}
	Let $2<t<\infty$. Assume $\overline\alpha$ is an aperture bigger than $\alpha$ and $\Delta= \Delta(Q_0,r)$ is a surface ball of radius $r$, then
	\[ \int_{\Delta} |S_r^{\overline\alpha} u(Q) |^t d\sigma(Q) \leq \int_{2(\alpha+1) \Delta} |S_{2(\alpha+1) r} u(Q)|^t d\sigma(Q). \]
\end{lemma}

Moreover, we have the following ``good-$\lambda$" inequality between $S_{r}u$ and the Carleson type function 
\begin{equation}
	Cu(Q) = \sup_{\Delta\ni Q} \frac{1}{\sigma(\Delta)}\iint_{\Gamma(Q)} \Carl{u} dX. \label{eq:Cfunc}
\end{equation} 

\begin{lemma}\label{lm:goodlambda}
	There exist an aperture $\overline \alpha>\alpha$ and a constant $C>0$, such that for any surface ball $\Delta = \Delta(Q_0, r)$ and any $\lambda,\gamma>0$
	\begin{align}
		& \sigma\Big(\big\{Q\in\Delta: S_r u(Q) > 2\lambda, Cu(Q) \leq \gamma\lambda\big\}\Big)  \leq C\gamma^2 \sigma\Big(\big\{Q\in 4\Delta: S_{4r}^{\overline\alpha}u(Q) >\lambda\big\}\Big) \label{goodlambda}
	\end{align}
\end{lemma}
\begin{proof}
	Let $\overline\alpha=\alpha+3$. Consider the open set $\mathcal{O} = \{ Q\in 4\Delta : S_{4r}^{\overline\alpha}u(Q) >\lambda\}$. 
	Similar to Lemma 3.7 and Lemma 6.2 in \cite{HACAD}, let $\cup_{k}\Delta_k$ be a Whitney decomposition of $\mathcal{O} $, such that 
	\[ \text{for each } k, \quad \Delta\big(Q_k,\frac{1}{24} d(Q_k)\big) \subset \Delta_k \subset \Delta\big(Q_k, \frac{1}{2} d(Q_k)\big). \]
	Here $Q_k\in \mathcal{O} $, and $d(Q_k) = \dist(Q_k, \mathcal{O}^{c} )>0$. 
	We claim that for all $\Delta_k$ such that $\Delta_k \cap \Delta \neq \emptyset$, we have
	\begin{equation}\label{eq:breakpart}
		\sigma\Big(\big\{Q\in\Delta_k : S_r u(Q) > 2\lambda, Cu(Q) \leq \gamma\lambda\big\}\Big) \leq C \gamma^2 \sigma(\Delta_k).
	\end{equation}
	This is clearly true if the left hand side is empty. Assume it is not empty, and
	\begin{equation}
		\text{there is some } Q'_k \in \big\{Q\in\Delta_k: S_r u(Q) > 2\lambda, Cu(Q) \leq \gamma\lambda\big\}. \label{Qkprime}
	\end{equation} 
		
	
	Note that 
	\begin{align*}
		d(Q_k) = \dist(Q_k,\mathcal{O}^c)  = \min\Big\{\dist\big(Q_k, \big\{Q\in 4\Delta: S_{4r}^{\overline\alpha}u(Q) \leq \lambda\big\}\big), \dist(Q_k, (4\Delta)^c) \Big\}, 
	\end{align*} 
	we need to consider two cases.
	
	\textit{Case 1.} Assume $Q_k$ is such that 
	\[ d(Q_k) = |Q_k - P_k|, \quad\text{for some~} P_k\in 4\Delta \text{~satisfying~} S^{\overline\alpha}_{4r} u(P_k) \leq \lambda. \]
	Let $Q\in\Delta_k $ be arbitrary, recall that $\Delta_k\subset\Delta(Q_k, d(Q_k)/2)$, hence
	\[ |Q-P_k| \leq |Q-Q_k| + |Q_k-P_k| < \frac 1 2 d(Q_k) + d(Q_k) = \frac{3}{2} d(Q_k). \]
	For any $X\in \Gamma_r(Q)$, we define the functions 
	\begin{equation}\label{def:uonetwo}
		u_1(X) = u(X) \chi_{\{\delta(X) \geq d(Q_k)/2\}}, \text{ and }u_2(X) = u(X) \chi_{\{\delta(X) < d(Q_k)/2\}}. 
	\end{equation}  
	Clearly $S_r u(Q) \leq S_r u_1(Q) + S_r u_2(Q)$.
	
	 If $X\in \Gamma_r(Q)$ is such that $\delta(X) \geq d(Q_k)/2$, we have
	\[ |X-P_k| \leq |X-Q| + |Q-P_k| < \alpha \delta(X) + \frac{3}{2} d(Q_k) \leq \overline{\alpha} \delta(X), \]
	and 
	\[ |X-P_k|\leq |X-Q| + |Q-P_k| < r + 3\delta(X) \leq 4r. \]
	In other words $X\in \Gamma^{\overline\alpha}_{4r}(P_k)$. Hence
	\begin{align}
		|S_r u_1(Q)|^2 & = \iint_{\Gamma_r (Q)\cap \{\delta(X) \geq d(Q_k)/2\}} |\nabla u|^2 \delta(X)^{2-n} dX \nonumber \\
		& \leq \iint_{\Gamma_{4r}^{\overline\alpha}(P_k)} |\nabla u|^2 \delta(X)^{2-n} dX \nonumber \\
		& = |S_{4r}^{\overline\alpha} u(P_k)|^2  \nonumber \\
		& \leq \lambda^2. \label{eq:uone}
	\end{align} 
	
	If $X\in \Gamma_r(Q)$ is such that $\delta(X) < d(Q_k)/2$, recall $Q \in \Delta_k\subset \Delta(Q_k, d(Q_k)/2)$ and \eqref{Qkprime}, we have 
	\begin{align*}
		|X-Q'_k| & \leq |X-Q|+|Q-Q_k|+|Q_k - Q'_k|\\
		& \leq \alpha\delta(X) + \frac 1 2 d(Q_k) + \frac 1 2 d(Q_k) \\
		& < \left(\frac{\alpha}{2} +1\right) d(Q_k).
	\end{align*}  
	Hence
	\begin{align}
		\int_{\Delta_k} S_r^2 u_2(Q) d\sigma & = \int_{\Delta_k} \iint_{\Gamma_r(Q) \cap \{\delta(X) < d(Q_k)/2\}} |\nabla u|^2 \delta(X)^{2-n} dX \nonumber \\
		 & \leq \iint_{B(Q'_k, \left(\frac{\alpha}{2} +1\right) d(Q_k))} \Carl{u} dX \nonumber \\
		& \leq |Cu(Q'_k)|^2 \cdot \sigma\left(\Delta\left(Q'_k, \left(\frac{\alpha}{2} +1\right) d(Q_k)\right)\right), \label{eq:tmp} 
	\end{align}
	where the Carleson type function is defined in \eqref{eq:Cfunc}. By \eqref{Qkprime}, we know that $Cu(Q'_k)\leq \gamma\lambda$. In addition, $\sigma$ is Ahlfors regular and $\Delta(Q_k,d(Q_k)/24) \subset \Delta_k$. Therefore it follows from \eqref{eq:tmp}
	\begin{equation}
		\int_{\Delta_k} S_r^2 u_2(Q) d\sigma  \leq \gamma^2 \lambda^2 \cdot C_2 \left(\frac{\alpha}{2}+1\right)^{n-1} d(Q_k)^{n-1} \leq C\gamma^2\lambda^2 \sigma(\Delta_k), \label{eq:secondu}
	\end{equation}
	where the constant $C$ only depends on the aperture $\alpha$ and the Ahlfors regular constants of $\sigma$.
	On the other hand, 
	\[ \int_{\Delta_k} S_r^2 u_2(Q) d\sigma \geq \lambda^2 \sigma\Big(\big\{Q\in\Delta_k: S_r u_2(Q) > \lambda\big\}\Big), \]
	hence $\sigma\Big(\big\{Q\in\Delta_k: S_r u_2(Q) > \lambda\big\}\Big) \leq C\gamma^2 \sigma(\Delta_k)$.
	
	Recall $S_r u_1(Q) \leq \lambda$ for all $Q\in\Delta_k$ (see \eqref{eq:uone}), therefore
	\begin{align*}
	    \sigma\Big(\big\{Q\in\Delta_k: S_r u(Q) > 2\lambda, Cu(Q) \leq \gamma\lambda \big\}\Big) & \leq \sigma\Big(\big\{Q\in\Delta_k: S_r u_2(Q) > \lambda, Cu(Q) \leq \gamma\lambda \big\}\Big) \\
		&  \leq \sigma\Big(\big\{Q\in\Delta_k: S_r u_2(Q) > \lambda\big\}\Big) \\
		&  \leq C\gamma^2 \sigma(\Delta_k). 
	\end{align*} 
	
	\textit{Case 2.} Assume $Q_k$ is such that $d(Q_k) = \dist(Q_k, (4\Delta)^c)$. 
	We only consider the $\Delta_k$'s such that $\Delta_k\cap\Delta \neq\emptyset$, and assume $R_k$ is a point in the intersection. In particular $R_k\in\Delta_k\subset \Delta(Q_k, d(Q_k)/2)$. So
	\[ \dist(R_k, (4\Delta)^c) \leq |R_k - Q_k| + \dist(Q_k, (4\Delta)^c) < \frac 1 2 d(Q_k) + d(Q_k) = \frac 3 2 d(Q_k). \]
	On the other hand, suppose $\dist(R_k, (4\Delta)^c) = |R_k - R|$ for some $R\in (4\Delta)^c$, then
	\[ \dist(R_k, (4\Delta)^c) \geq |R- Q_0|- |R_k - Q_0|> 4r - r = 3r. \]
	It follows that $d(Q_k)/2> r $. In particular, for any $Q\in\Delta_k$ and $X\in\Gamma_r(Q)$, we have $\delta(X) \leq r < d(Q_k)/2$.
	In other words, $u = u_2$ (see \eqref{def:uonetwo}). Similar to \eqref{eq:secondu}, one can show
	\[ \int_{\Delta_k} S_r^2 u(Q) d\sigma = \int_{\Delta_k} S_r^2 u_2(Q) d\sigma \leq C\gamma^2 \lambda^2 \sigma(\Delta_k). \]
	Therefore
	\begin{align*}
	 \sigma\Big(\big\{Q\in\Delta_k: S_r u(Q) > 2\lambda, Cu(Q) \leq \gamma \lambda\big\}\Big) & \leq \sigma\Big(\big\{Q\in\Delta_k: S_r u(Q) > 2\lambda \big\}\Big) \leq \frac{C \gamma^2 \sigma(\Delta_k)}{4} . 
	\end{align*} 
	This finishes the proof of \eqref{eq:breakpart}.
	
	Summing up \eqref{eq:breakpart} for $\Delta_k$'s such that $\Delta_k\cap \Delta \neq \emptyset$, we obtain 
	\begin{align}
		& \sigma\left(\left\{Q\in \bigcup\limits_{k: \Delta_k \cap \Delta \neq \emptyset} \Delta_k : S_r u(Q) >2\lambda, Cu(Q) \leq \gamma\lambda \right\}\right) \nonumber \\
		& \qquad \qquad \qquad \leq C\gamma^2 \sigma\left(\bigcup\limits_{k} \Delta_k\right) \leq C\gamma^2 \sigma\Big(\big\{Q\in 4\Delta: S^{\overline \alpha}_{4r} u(Q) > \lambda\big\}\Big).\label{eq:sum}
	\end{align}
	The last inequality is because $\{\Delta_k\}$ is a Whitney decomposition of $\{Q\in 4\Delta: S_{4r}^{\overline\alpha} u(Q)>\lambda\}$. It also implies
	\[ \bigcup\limits_{k: \Delta_k\cap \Delta \neq \emptyset} \Delta_k \supset \big\{Q\in\Delta: S_{4r}^{\overline\alpha} u(Q) >\lambda\big\}. \]
	Therefore
	\begin{align}
		& \left\{Q\in \bigcup\limits_{k: \Delta_k \cap \Delta \neq \emptyset} \Delta_k : S_r u(Q) >2\lambda, Cu(Q) \leq \gamma\lambda \right\} \nonumber \\
		& \qquad \qquad \supset \left\{Q\in\Delta: S_{4r}^{\overline\alpha}u >\lambda \text{~and~} S_r u(Q) >2\lambda, Cu(Q) \leq \gamma\lambda \right\} \nonumber \\
		& \qquad \qquad = \big\{Q\in\Delta: S_r u(Q) > 2\lambda, Cu(Q) \leq \gamma\lambda\big\} \label{eq:setcomp}
	\end{align} 
	For the last equality, we use $\overline\alpha>\alpha$ and thus $\big\{ S_r u>2\lambda \big\} \subset \big\{ S_{4r}^{\overline\alpha}u >\lambda \big\}$.
    Combining \eqref{eq:sum} and \eqref{eq:setcomp}, we get
	\begin{align*}
		\sigma\Big(\big\{Q\in\Delta: S_r u > 2\lambda, Cu \leq \gamma\lambda\big\}\Big) \leq C\gamma^2 \sigma\Big(\big\{Q\in 4\Delta: S^{\overline\alpha}_{4r} u(Q) >\lambda\big\}\Big). 
	\end{align*} 
	\end{proof}
	
	By Lemma \ref{lm:goodlambda},
	\begin{align}
		\int_{\Delta} |S_r u|^t d\sigma 
		& = t \int_0^{\infty} \lambda^{t-1} \sigma\Big(\big\{Q\in\Delta: S_r u > \lambda, Cu \leq \gamma\lambda/2 \big\}\Big) d\lambda \nonumber \\
		& \qquad \qquad \quad + t \int_0^{\infty} \lambda^{t-1} \sigma\Big(\big\{Q\in\Delta: S_r u > \lambda, Cu > \gamma \lambda/2\big\}\Big) d\lambda \nonumber \\
		& \leq t \int_0^{\infty} \lambda^{t-1} \cdot C\gamma^2 \sigma\Big(\big\{Q\in 4\Delta: S^{\overline\alpha}_{4r} u >\lambda/2\big\}\Big) d\lambda \nonumber \\
		& \qquad \qquad \quad + t \int_0^{\infty} \lambda^{t-1} \sigma\Big(\big\{Q\in\Delta: Cu > \gamma \lambda/2\big\}\Big) d\lambda \nonumber \\
		& = C\gamma^2 2^t \int_{4\Delta} |S^{\overline\alpha}_{4r}u|^t d\sigma + \left(\frac{2}{\gamma}\right)^{t} \int_{\Delta} |Cu|^t d\sigma \nonumber \\
		& \leq C' \gamma^2 \int_{4\Delta} |S^{\overline\alpha}_{4r}u|^t d\sigma + \left(\frac{2}{\gamma}\right)^{t} |\mathscr{C}(u)|^t \sigma(\Delta). \label{eq:boundq}
	\end{align}
	Here $\mathscr{C}(u) = \sup_{\Delta\subset \pO} \frac{1}{\sigma(\Delta)} \iint_{T(\Delta)} \Carl{u} dX$
	stands for the Carleson measure defined by the function $u$, and by definition $Cu(Q) \leq \Carleson(u)$ for all $Q\in \pO$. 
	Apply Lemma \ref{lemma:diffapper} to the right hand side of \eqref{eq:boundq}, it becomes
	\begin{equation}\label{eq:beforeavrg}
		\int_{\Delta} |S_r u|^t d\sigma \leq C'' \gamma^2 \int_{A\Delta} |S_{Ar} u|^t d\sigma + \left(\frac{2}{\gamma}\right)^{t} |\Carleson(u)|^t \sigma(\Delta),
	\end{equation}
	where $A= 8(\alpha+1)$ is a constant and $A\Delta=\Delta(Q_0,Ar) $. If the radius $r$ is such that $Ar < \diam \Omega$, we can rewrite the above inequality in the following form:
	\begin{equation}
		\fint_{\Delta} |S_r u|^t d\sigma \leq \widetilde C \gamma^2 \fint_{A\Delta} |S_{Ar} u|^t d\sigma + \left(\frac{2}{\gamma}\right)^{t} |\Carleson(u)|^t.
	\end{equation}
    Pick $\gamma$ (depending on $\alpha$) so that $\widetilde C\gamma^2 = 1/4$. Fix such $\gamma$ fixed, denote $C_1 = (2/\gamma)^t $, then
	\begin{equation}\label{ineq:avrg}
		\fint_{\Delta} |S_r u|^t d\sigma \leq \frac 1 4 \fint_{A\Delta} |S_{Ar} u|^t d\sigma + C_1 |\Carleson(u)|^t.
	\end{equation}
	Theorem 6.1 in \cite{HACAD} states the following global estimate 
	\begin{equation}
		\int_{\pO} |Su|^t d\sigma \leq C \int_{\pO} |Cu|^t d\sigma \leq C |\Carleson(u)|^t \sigma(\pO) . \label{eq:globalSC}
	\end{equation} 
	We claim the \textquotedblleft contraction\textquotedblright estimate \eqref{ineq:avrg}, together with the global estimate \eqref{eq:globalSC} 
	implies
	\begin{align}
	   \sup_{\Delta_r\subset \pO} \left(\fint_{\Delta_r} |S_r u|^t d\sigma\right)^{1/t} & \leq C\cdot \Carleson(u)  = C \sup_{\Delta\subset \pO } \frac{1}{\sigma(\Delta)} \iint_{T(\Delta)} \Carl{u} dX.\label{squareCarl}
	\end{align}

	Firstly, for an arbitrary $r>0$, let $k$ be a positive integer such that $A^k r <\diam\Omega \leq A^{k+1} r$, then $\sigma(A^k \Delta) \approx \sigma(\pO) $. (The constants only depends on $A$ and the Ahlfors regularity of $\sigma$. In particular they do not depend on $r$ or $Q_0$.) Apply \eqref{eq:beforeavrg} to $A^k \Delta$, we get
\begin{align}
	\int_{A^k\Delta} |S_{A^k r} u|^t d\sigma & \leq C''' \int_{A^{k+1}\Delta} |S_{A^{k+1}r} u|^t d\sigma + C_1 |\Carleson(u)|^t \sigma(A^k \Delta) \nonumber \\
	& \leq C''' \int_{\pO} |Su|^t d\sigma + \widetilde C_1 |\Carleson(u)|^t \sigma(\pO) \nonumber \\
	& \leq C_2 |\Carleson(u)|^t \sigma(\pO). \qquad \qquad \text{by } \eqref{eq:globalSC} \nonumber
\end{align} 
Hence
\begin{equation}\label{eq:almostglobalSC}
	\fint_{A^k\Delta} |S_{A^k r} u|^t d\sigma \leq C_2 |\Carleson(u)|^t.
\end{equation}

To simplify the notations, we write $a_{r} = \fint_{\Delta} |S_r u|^t d\sigma$ and $B=\max\{C_1, C_2\} \cdot |\Carleson(u)|^t$,
	where $C_1$ and $C_2$ are the constants in \eqref{ineq:avrg} and \eqref{eq:almostglobalSC}.
	Hence \eqref{ineq:avrg} and \eqref{eq:almostglobalSC} become
\begin{equation}\label{simplified}
	a_r \leq \frac{1}{ 4} a_{Ar} + B \qquad \text{if~} Ar<\diam\Omega. 
\end{equation}
\begin{equation}
	a_{A^k r} \leq B \quad \text{where~} A^k r<\diam\Omega \leq A^{k+1}r .
\end{equation}
Induction on \eqref{simplified}, we obtain
\begin{align}
	a_r  \leq \frac{1}{4} a_{Ar} + B \leq \frac{1}{4^k} a_{A^k r} + \left(1 + \frac{1}{ 4} + \cdots + \frac {1}{4^{k-1}}\right) B \leq \frac{7}{3}B. \label{eq:CCLsimplified}
\end{align}
In other words, 
\begin{equation}
	\fint_{\Delta(Q_0,r)} |S_r u|^t d\sigma \leq C|\Carleson(u)|^t, \quad\text{with the constant } C= \max\{C_1, C_2\} \cdot 7/3.
\end{equation}
This holds for arbitrary $Q_0\in\pO$ and $r\in (0,\diam\Omega)$, so \eqref{squareCarl} follows. This finishes the proof of the theorem \ref{tm:squareCarleson}, hence the conclusion \eqref{reverseCarl} follows. 

\section{Acknowledgement}

The author was partially supported by NSF DMS grants 1361823 and 1500098. The author wants to thank her advisor Prof. Tatiana Toro for introducing her to this area and the enormous support during the work on this paper. The author also wants to thank Prof. Hart Smith for his support, and thank the referee for the careful reading and helpful suggestions.

\appendix
\appendixpage

\section{Proof of Lemma \ref{lm:fe}: Properties of $f_{\epsilon}$}\label{sect:feprop}


The function $f_{\epsilon}$ as in \eqref{deffe} is well defined because
\begin{equation}\label{eq:vel}
	\int_{y\in\pO} \ve(x-y)d\sigma(y) \geq \frac{1}{\epsilon^{n-1}} \int_{y\in \Delta(x, \frac{\epsilon}{2})} d\sigma(y) \geq C_1 > 0.
\end{equation} 
We also have
\begin{equation}\label{eq:veu}
	\int_{y\in\pO} \ve(x-y)d\sigma(y) \leq \frac{1}{\epsilon^{n-1}} \int_{y\in \Delta(x, \epsilon)} d\sigma(y) \leq C_2 .
\end{equation} 
 The constants $C_1$ and $C_2$ are independent of $\epsilon$.

\textit{Proof of (1).} For any surface ball $\Delta_0=\Delta(x_0, r_0)$, we denote $\Delta_0^{\epsilon} = \Delta(x_0,r_0+\epsilon)$. Since $f$ is supported in $2\Delta$, each $f_{\epsilon}$ is supported in $\left(2\Delta\right)^{\epsilon}$. Thus all $f_{\epsilon}$'s are supported in $3\Delta$ if $\epsilon< r$, the radius of $\Delta$.

Note that
\begin{align}
	& \left|\int_{\pO} \ve(x-y) d\sigma(y) -\int_{\pO} \ve(\tilde x-y) d\sigma(y) \right| \nonumber \\
	= & \left| \int_{\pO} \int_0^1 \frac{d}{ds} \ve((1-s)\tilde x+s x-y) ds d\sigma(y) \right| \nonumber \\
	\leq & \frac{|x-\tilde x|}{\epsilon^n} \int_0^1 \int_{y\in \pO} \left| \nabla\varphi\left(\dfrac{(1-s)\tilde x+s x-y}{\epsilon} \right) \right| d\sigma(y) ds. \label{eq:lineinteg}
\end{align}
Since $\|\nabla \varphi\|_{L^{\infty}} \leq C$, for any $w\in\mathbb{R}^n$ we have
\begin{equation}\label{eq:lineintegtmp}
	\int_{y\in\pO} \left|\nabla\varphi\left(\frac{w-y}{\epsilon}\right)\right| d\sigma(y) \leq C \sigma \left( B(w,\epsilon)\cap\pO \right) \leq C\epsilon^{n-1}. 
\end{equation} 
Combining \eqref{eq:lineinteg} and \eqref{eq:lineintegtmp},
\[ \left|\int_{\pO} \ve(x-y) d\sigma(y) -\int_{\pO} \ve(\tilde x-y) d\sigma(y) \right| \leq C\frac{|x-\tilde x|}{\epsilon}, \]
so for any $\epsilon$ fixed, the map $x\in\pO \mapsto \int_{\pO} \ve(x-y) d\sigma(y)$ is continuous. Since $0\leq f \leq 1$, we can prove similarly $\int_{\pO} f(y) \ve(x-y) d\sigma(y)$ is also continuous. Thus $f_{\epsilon}(x)$ is continuous.

\textit{Proof of (2).} 
Fix $\epsilon>0$. Let $\widetilde \Delta=\Delta(x_0, r_0)$ be an arbitrary surface ball. Let $\lambda$ be a real number to be determined later. We consider two cases.

\noindent \textit{Case 1.} If $r_0 \geq \epsilon/2$, by the definition \eqref{deffe} and the estimates \eqref{eq:vel}, \eqref{eq:veu},
\begin{align*}
	\int_{\dt} \left| f_{\epsilon}(x) -\lambda\right| d\sigma(x) & \leq \frac{1}{C_1} \int_{\dt}\left| \int_{\pO} f(y) \ve(x-y) d\sigma(y) - \lambda \int_{\pO} \ve(x-y) d\sigma(y) \right| d\sigma(x) \\
	& \leq \widetilde C_1 \int_{x\in \dt} \int_{y\in\Delta(x,\epsilon) } |f(y) -\lambda| \ve(x-y) d\sigma(y) d\sigma(x) \\
	& \leq \widetilde C_1 \int_{y\in \dt^{\epsilon}} |f(y) -\lambda| \int_{x\in\pO} \ve(x-y) d\sigma(x) d\sigma(y) \\
	& \leq \widetilde C_1 C_2 \int_{y\in \dt^{\epsilon}} |f(y) -\lambda| d\sigma(y) \\
	& \leq C' \sigma(\dt^{\epsilon}) \|f\|_{BMO(\sigma)} .
\end{align*}
The last inequality is true if we choose $\lambda = \lambda(\dt, \epsilon) $ be the constant satisfying $\fint_{\dt^{\epsilon}} |f(\cdot) - \lambda| d\sigma \leq \|f\|_{BMO(\sigma)}$.
Thus
\begin{align*}
	\fint_{\dt} \left| f_{\epsilon}(x) -\lambda\right| d\sigma(x) \lesssim \dfrac{\sigma(\dt^{\epsilon})}{\sigma(\dt)} \|f\|_{BMO(\sigma)} \lesssim \left(\frac{r_0+\epsilon}{r_0}\right)^{n-1} \|f\|_{BMO(\sigma)} \lesssim \|f\|_{BMO(\sigma)}.
\end{align*} 

\noindent \textit{Case 2.} If $r_0<\epsilon/2$, by the definition \eqref{deffe} and the estimate \eqref{eq:vel},
\begin{align}
	\int_{\dt} \left| f_{\epsilon}(x) -\lambda\right| d\sigma(x) & \leq \widetilde C_1 \int_{x\in \dt} \int_{y\in\Delta(x,\epsilon) } |f(y) -\lambda| \ve(x-y) d\sigma(y) d\sigma(x) \nonumber \\
	& \leq \widetilde C_1 \int_{y\in \dt^{\epsilon}} |f(y) -\lambda| \int_{x\in \dt } \ve(x-y) d\sigma(x) d\sigma(y). \label{eq:fesmallr}
\end{align}
Note
\[ \int_{x\in \dt } \ve(x-y) d\sigma(x) \leq \frac{1}{\epsilon^{n-1}} \int_{x\in\dt} d\sigma(x)=\frac{\sigma(\dt)}{\epsilon^{n-1}} , \]
it follows from \eqref{eq:fesmallr} that
\begin{align*}
	\fint_{\dt} \left| f_{\epsilon}(x) -\lambda\right| d\sigma(x) & \lesssim \frac{1}{\sigma(\dt)} \cdot \frac{\sigma(\dt)}{\epsilon^{n-1}} \int_{y\in \dt^{\epsilon}} |f(y) -\lambda| d\sigma(y) \\
    & \lesssim	\dfrac{\sigma(\dt+\epsilon)}{\epsilon^{n-1} } \|f\|_{BMO(\sigma)} \\
	& \lesssim \|f\|_{BMO(\sigma)}.
\end{align*}  

We have proved the following: for any $\epsilon$ and any surface ball $\dt$, one can find a constant $\lambda=\lambda(\dt, \epsilon)$ such that $\fint_{\dt} \left| f_{\epsilon}(x) -\lambda\right| d\sigma(x) \leq C \|f\|_{BMO(\sigma)}$. The constant $C$ does not depend on either $\epsilon$ or $\dt$, therefore $\|f_{\epsilon}\|_{BMO(\sigma)} \leq C\|f\|_{BMO(\sigma)}$ for all $\epsilon$.

\textit{Proof of (3).} Fix $x\in\pO$. If $f(x)=0$, then obviously $f(x) \leq \liminf_{\epsilon\to 0} f_{\epsilon}(x)$. For any arbitrary $\lambda>0$ such that $\lambda < f(x)$, there exists $\epsilon_0>0$ such that $f(x) > \lambda + \epsilon_0$. It means
\[ \sup_{\Delta'\ni x} \frac{\sigma(E\cap \Delta')}{\sigma(\Delta')} = M_{\sigma}\chi_E(x) > e^{\frac{\lambda + \epsilon_0 - 1}{\delta}}. \]
In particular, there is some surface ball $\Delta' \ni x$ such that
\[ \frac{\sigma(E\cap \Delta')}{\sigma(\Delta')} > e^{\frac{\lambda + \epsilon_0 - 1}{\delta}}. \]
Then for any point $y\in\Delta'$, we also have $M_{\sigma}\chi_E(y) > \exp{(\lambda+\epsilon_0-1)/\delta}$ and thus $f(y) > \lambda + \epsilon_0$. Consider all $f_{\epsilon}$ with $\epsilon< \dist(x, \pO\setminus \Delta')$, we have $\Delta(x,\epsilon)\subset \Delta'$, hence
\begin{align*}
	f_{\epsilon}(x) = \frac{ \int_{y\in \Delta(x,\epsilon)} f(y) \varphi_{\epsilon} (x-y) d\sigma(y)}{\int_{y\in\Delta(x,\epsilon)} \ve(x-y)d\sigma(y)} \geq (\lambda+\epsilon_0) \frac{ \int_{y\in \Delta(x,\epsilon)} \varphi_{\epsilon} (x-y) d\sigma(y)}{\int_{y\in\Delta(x,\epsilon)} \ve(x-y)d\sigma(y)} = \lambda + \epsilon_0. 
\end{align*}
Therefore $\liminf_{\epsilon\to 0} f_{\epsilon}(x) > \lambda $ for all $\lambda < f(x)$. Thus $\liminf_{\epsilon\rightarrow 0} f_{\epsilon}(x) \geq f(x)$. \qed

\section{Properties of the truncated square function}\label{sect:trsq}
\subsection{Proof of Lemma \ref{lemma:open}}
	Assume $Q\in\pO$ satisfies $S^2_r u(Q) = \iint_{\Gamma_r(Q)} |\nabla u|^2 \delta(X)^{2-n} dX > \lambda^2$ and is finite, then there exists $\eta<r$ such that
	\[ \iint_{\Gamma_r(Q)\setminus B(Q,\eta) } |\nabla u|^2 \delta(X)^{2-n} dX > \left(\frac{ S_r u(Q) +  \lambda}{2}\right)^2. \]
	Fix $\eta$, we claim there exists $\epsilon>0$ such that $S_r u(P)>\lambda$ for any $P\in B(Q,\epsilon\eta)\cap \pO$. In fact,
	\begin{align}
		\left|\iint_{\Gamma_r(Q)\setminus B(Q,\eta) } |\nabla u|^2 \delta(X)^{2-n} dX - \iint_{\Gamma_r(P)\setminus B(P,\eta) } |\nabla u|^2 \delta(X)^{2-n} dX\right| \nonumber \\ \leq
		 \iint_{D} |\nabla u|^2 \delta(X)^{2-n} dX \label{eq:boundbysetdiff}
	\end{align}
	where $D$ is the set difference between $\Gamma_r(Q)\setminus B(Q,\eta) $ and $\Gamma_r(P)\setminus B(P,\eta) $. 
	
	Assume $X\in \Gamma_r(Q)\setminus B(Q,\eta) $, then $|X-Q|\leq \alpha \delta(X)$ and $\eta \leq |X-Q|<r$.	In particular $\eta \leq \alpha \delta(X)$.
	If in addition $X\notin \Gamma_r(P)\setminus B(P,\eta) $ for some $P \in B(Q,\epsilon\eta)$, then $X$ falls in one of the following three categories:
	\begin{itemize}
		\item $|X-P|< \eta$, then $|X-Q|\leq |X-P|+|P-Q| <(1+\epsilon)\eta$, in particular $\eta\leq |X-Q| <(1+\epsilon)\eta $;
		\item $|X-P|\geq r$, then $|X-Q|\geq |X-P|-|P-Q| > r-\epsilon \eta$, in particular $r-\epsilon\eta<|X-Q|<r$;
		\item $|X-P|> \alpha\delta(X)$, then $|X-Q|\geq |X-P|-|P-Q|> (1-\epsilon) \alpha\delta(X)$, in particular $(1-\epsilon) \alpha\delta(X) < |X-Q| \leq \alpha\delta(X)$.
	\end{itemize}
	Similarly, the points in $\left(\Gamma_r(P)\setminus B(P,\eta)\right) \setminus \left(\Gamma_r(Q)\setminus B(Q,\eta) \right)$ also fall in three categories, just with $Q$ replaced by $P$. Therefore $D$, the set difference between $\left(\Gamma_r(Q)\setminus B(Q,\eta)\right) \setminus \left(\Gamma_r(P)\setminus B(P,\eta) \right)$ and $\left(\Gamma_r(P)\setminus B(P,\eta)\right) \setminus \left(\Gamma_r(Q)\setminus B(Q,\eta) \right)$, is contained in the union of three sets (corresponding to the above three cases):	
	\[ V_1 = \big\{X\in\Omega: (1-\epsilon)\eta < |X-Q|<(1+2\epsilon)\eta, ~\delta(X) \geq \eta/\alpha \big\}\]
	\[ V_2 = \big\{X\in\Omega: r-2\epsilon\eta< |X-Q| < r+\epsilon\eta, ~\delta(X) \geq \eta/\alpha \big\}\]
	\[ V_3 = \big\{X\in\Omega:(1-2\epsilon) \alpha\delta(X) < |X-Q| \leq (1+\epsilon)\alpha\delta(X), ~\delta(X) \geq \eta/\alpha \big\}.\]
	 Since $\delta(X) \geq \eta/\alpha$ in $D$,
	\begin{equation}
		\iint_{D} |\nabla u|^2 \delta(X)^{2-n} dX \leq \left(\frac{\alpha}{\eta}\right)^{n} \iint_{V_1\cup V_2 \cup V_3} |\nabla u|^2 \delta(X)^2 dX. \label{eq:setdiff}
	\end{equation}  
	
	Note that $u\in W^{1,2}(\Omega)$, we have
	\[ \iint_{\Omega} |\nabla u|^2 \delta(X)^2 dX \leq \diam(\Omega)^2 \iint_{\Omega} |\nabla u|^2 dX <\infty. \] 
	Hence the integral $ \iint_{V_1\cup V_2 \cup V_3} |\nabla u|^2 \delta(X)^2 dX$ is small as long as the Lebesgue measures of $V_1$, $V_2$ and $V_3$ are small enough. Both $V_1$ and $V_2$ are contained in annuli of radius $3\epsilon\eta$, so their Lebesgue measures are small if we choose $\epsilon$ small enough (depending on $\eta$). Rewrite $V_3$ as
	\[ V_3 = \bigg\{X\in\Omega: \frac{1}{ (1+\epsilon) \alpha}< \frac{\delta(X)}{|X-Q|} \leq \frac{1}{(1-2\epsilon)\alpha}, ~ \delta(X) \geq \frac{\eta}{\alpha}\bigg\}. \]
	Away from $Q$, say in $\Omega \setminus B(Q,\eta/2)$, the function $F(X) = \delta(X)/|X-Q|$ is Lipschitz, and $0\leq F\leq 1$. Choose $\epsilon<1/4$, then for any $X\in V_3$, $|X-Q|\geq (1-2\epsilon)\alpha \delta(X) \geq \eta/2$.
	So $V_3\subset \Omega\setminus B(Q,\eta/2)$ and thus $F$ is Lipschitz on $V_3$. By the coarea formula,
	\begin{equation} \mathcal{H}^n(V_3) =  \int_{\frac{1}{(1+\epsilon)\alpha} }^{\frac{1}{(1-2\epsilon)\alpha}} \int_{F^{-1}(t)} \frac{1}{JF} \chi_{V_3} d\mathcal{H}^{n-1} dt.	\label{eq:Vmeasure}
	\end{equation}
	On the other hand, 
	\begin{align*}
		\int_{0}^1 \int_{F^{-1}(t)} \frac{1}{JF} \chi_{V_3} d\mathcal{H}^{n-1} dt  \leq \int_{0}^1 \int_{F^{-1}(t)} \frac{1}{JF} \chi_{\Omega\setminus B(Q,\eta/2) } d\mathcal{H}^{n-1} dt = \mathcal{H}^n\big(\Omega \setminus B(Q,\eta/2) \big)  
	\end{align*} 
	is finite.
	Therefore by \eqref{eq:Vmeasure}, we may choose $\epsilon$ small enough (depending on $\alpha$) such that $\mathcal{H}^n(V_3)$ is small, which in turn implies $\iint_{V_3} |\nabla u|^2 \delta(X)^2 dX $ is small. To sum up, we have shown that one can choose $\epsilon = \epsilon(\delta,\alpha,\eta,r)$ small enough such that 
	\[ \left(\frac{\alpha}{\eta}\right)^{n} \iint_{V_1\cup V_2 \cup V_3} |\nabla u|^2 \delta(X)^2 dX < \delta < \left(\dfrac{S_r u(Q) + \lambda}{2} \right)^2 -\lambda^2. \]
	Therefore we conclude from \eqref{eq:boundbysetdiff} and \eqref{eq:setdiff} that
	\begin{align*}
		& \iint_{\Gamma_r(P)\setminus B(P,\eta) } |\nabla u|^2 \delta(X)^{2-n} dX \\
		& \qquad\qquad \geq \iint_{\Gamma_r(Q)\setminus B(Q,\eta) } |\nabla u|^2 \delta(X)^{2-n} dX - \iint_{D } |\nabla u|^2 \delta(X)^{2-n} dX \\
		& \qquad\qquad > \left(\frac{ S_r u(Q) +  \lambda}{2}\right)^2 - \delta \\
		& \qquad\qquad >\lambda^2.
	\end{align*}
	Hence $S_r u(P) \geq \left( \iint_{\Gamma_r(P)\setminus B(P,\eta) } |\nabla u|^2 \delta(X)^{2-n} dX \right)^{1/2} >\lambda$,
	for all $P\in B(Q,\epsilon\eta)\cap\pO$. This finishes the proof that $\big\{Q\in\pO: S_r u(Q)>\lambda\big\}$ is open in $\pO$. \qed

\subsection{Proof of Lemma \ref{lemma:diffapper}}
	We prove the estimate by duality: let $r$ be the conjugate of $q/2$, namely $1/r+2/q = 1$, then
	\begin{align}\label{eq:normbyduality}
		\left(\int_{\Delta} |S_r^{\overline\alpha} u(Q) |^q d\sigma(Q)\right)^{2/q} = \sup \left\{\int_{\Delta} |S_r^{\overline\alpha}u(Q)|^2 \psi(Q) d\sigma(Q): \|\psi\|_{L^r(\Delta)} = 1\right\}.
	\end{align}
	Recall $\Delta=\Delta(Q_0,r)$. Extending $\psi$ to all of $\pO$ by setting it to zero outside of $\Delta$.
	For any $X$, let $Q_X\in \pO$ be a boundary point such that $|X-Q_X| = \delta(X)$. By Fubini's theorem,
	\begin{align}
		\int_{\Delta} |S_r^{\overline\alpha}u(Q)|^2 \psi(Q) d\sigma(Q) &  = \int_{\Delta} \left( \iint_{\Gamma^{\overline\alpha}_r(Q)} |\nabla u|^2 \delta(X)^{2-n} dX \right) \psi(Q) d\sigma(Q) \nonumber \\
		& \leq \iint_{B(Q_0,2r)\cap \Omega} |\nabla u|^2 \delta(X)^{2-n} \int_{|Q-Q_X|\leq (\overline\alpha+1)\delta(X)} \psi(Q) d\sigma(Q) dX \nonumber \\
		& \lesssim \iint_{B(Q_0,2r)\cap \Omega} |\nabla u|^2 \delta(X) A_{(\overline\alpha+1)\delta(X)}\psi(Q_X) dX, \label{eq:duality}
	\end{align}
	where $A_{s}\psi(Q)$ is defined as $A_{s}\psi(Q) = \frac{1}{s^{n-1}} \int_{\Delta(Q,s)} \psi(P) d\sigma(P)$.
	Let $\beta>1$, simply calculations show that
	\begin{align*}
		A_{s} \left( A_{\beta s} \psi\right) (Q) & = \frac{1}{s^{n-1}} \int_{\Delta(Q, s)} \left( \frac{1}{(\beta s)^{n-1}} \int_{\Delta(P,\beta s)} \psi(P') d\sigma(P') \right) d\sigma(P) \\
		& \geq \frac{1}{s^{n-1}} \int_{\Delta(Q, s)} \left( \frac{1}{(\beta s)^{n-1}} \int_{\Delta(Q,(\beta-1)s)} \psi(P') d\sigma(P') \right) d\sigma(P) \\
		& \gtrsim \left( \frac{(\beta-1)s}{\beta s}\right)^{n-1} A_{(\beta-1)s} \psi(Q).
	\end{align*}
	Let $s=(\alpha-1) \delta(X)$, $\beta-1=\left(\overline\alpha+1\right)/\left(\alpha-1\right) $, then
	\[ A_{(\overline\alpha+1)\delta(X)}\psi(Q) \lesssim_{\alpha,\overline\alpha} A_{(\alpha-1)\delta(X)}\left(A_{\beta s}\psi\right) (Q) \lesssim A_{(\alpha-1) \delta(X)}M\psi(Q). \]
	For the last inequality, we use $|A_{\beta s}\psi(Q)| \leq C \left(M\psi(Q)\right)$, where $M\psi$ is the Hardy-Littlewood maximal function of $\psi$ with respect to $\sigma$, and the constant $C$ only depend on the Ahlfors regularity of $\sigma$.
	Thus it follows from \eqref{eq:duality} that
	\begin{align} 
	\int_{\Delta} |S_r^{\overline\alpha}u(Q)|^2 \psi(Q) d\sigma(Q) 
	& \lesssim \iint_{B(Q_0,2r)\cap \Omega} |\nabla u|^2 \delta(X) A_{(\alpha-1)\delta(X)}M\psi(Q_X) dX \nonumber \\
	& \lesssim \iint_{B(Q_0,2r)\cap \Omega} |\nabla u|^2 \delta(X)^{2-n} \left( \int_{\Delta(Q_X, (\alpha-1)\delta(X))}M\psi(Q)d\sigma(Q)\right)  dX .\label{eq:squarediffangle}
	\end{align}
	By switching the order of integration,
	we can bound the right hand side by:	
	\begin{align}
	\int_{\Delta} |S_r^{\overline\alpha}u(Q)|^2 \psi(Q) d\sigma(Q) & \lesssim \int_{\Delta(Q_0, 2(\alpha+1)r)} M\psi(Q) \iint_{\Gamma_{2\alpha r} (Q)} |\nabla u|^2\delta(X)^{2-n} dX d\sigma(Q) \nonumber \\
	& = \int_{\Delta(Q_0, 2(\alpha+1)r)} M\psi(Q) |S_{2\alpha r} u(Q)|^2 d\sigma(Q) \nonumber \\
	& \leq \|M\psi\|_{L^r(\Delta(Q_0,2(\alpha+1)r))} \left(\int_{\Delta(Q_0, 2(\alpha+1)r)}|S_{2\alpha r} u(Q)|^q d\sigma(Q)\right)^{2/q}. \label{eq:duality1}
	\end{align}
	Since $1<r<\infty$, we have
	\begin{equation}\label{eq:bdduality}
		\|M\psi\|_{L^r(\Delta(Q_0,2(\alpha+1)r))} \leq C \|\psi\|_{L^r(\Delta(Q_0,2(\alpha+1)r))} = C \|\psi\|_{L^r(\Delta)}= C.
	\end{equation} 
	By \eqref{eq:duality1}, \eqref{eq:bdduality} and the definition \eqref{eq:normbyduality}, we conclude
	\begin{align*}
		\int_{\Delta} |S_r^{\overline\alpha} u(Q) |^q d\sigma(Q) & \leq C\int_{\Delta(Q_0, 2(\alpha+1)r)}|S_{2\alpha r} u(Q)|^q d\sigma(Q) \\
		& \leq C\int_{\Delta(Q_0, 2(\alpha+1)r)}|S_{2(\alpha+1) r} u(Q)|^q d\sigma(Q) .
	\end{align*} 
	This finishes the proof of Lemma \ref{lemma:diffapper}. \qed

\Addresses

\end{document}